\documentclass[12pt,a4paper,openright,twoside]{article}
\usepackage[english]{babel}
\usepackage[latin1]{inputenc}
\usepackage[toc,page]{appendix}
\usepackage{color,soul}
\usepackage{graphicx}
\usepackage{newlfont}
\usepackage{amssymb}
\usepackage{amsmath}
\usepackage{latexsym}
\usepackage{amsthm}
\usepackage[table]{xcolor}
\usepackage{tikz-cd}
\usepackage{afterpage}
\usepackage{tikz}
\usetikzlibrary{tikzmark}
\usetikzlibrary{positioning}
\usepackage{pdflscape}
\usepackage{rotating}
\usepackage{float}
\usepackage{authblk}
\usepackage{enumitem}
\setlist{noitemsep}

\theoremstyle{plain}                    
\newtheorem{teo}{Theorem}[section]
\newtheorem{prop}[teo]{Proposition}    
\newtheorem{cor}[teo]{Corollary}       
\newtheorem{lem}[teo]{Lemma} 
\theoremstyle{definition}               
\newtheorem{defin}{Definition}
      
\theoremstyle{remark}                   
          
\title{Rational cohomology of the moduli space of trigonal curves of genus 5}
\author{Angelina Zheng}

\date{}

\begin{document}
	\maketitle
	\begin{abstract}
		We compute the rational cohomology of the moduli space of trigonal curves of genus 5. We do this by considering their natural embedding in the first Hirzebruch surface and by using Gorinov-Vassiliev's method.
		
	\end{abstract}
	\section{Introduction}
	The aim of this work is to compute the cohomology with rational coefficients of the moduli space of trigonal curves of genus 5. We will consider algebraic curves defined over the complex field $\mathbf{C}.$\\
	An algebraic curve is said to be trigonal if it is not hyperelliptic and it admits a $g^1_3.$ \\
	We define the moduli space $\mathcal{T}_g$
	as the locus of trigonal curves in $\mathcal{M}_g,$ the moduli space of smooth curves of genus $g.$
	For $g=3,4$ the rational cohomology of $\mathcal{T}_g$ can be then deduced from that of $\mathcal{M}_g$, which was computed by Looijenga in \cite{Loo} for $g=3,$ and by Tommasi in \cite{Tom} for $g=4.$ In fact, in these two cases $\mathcal{T}_g$ coincides with $\mathcal{M}_g\backslash\mathcal{H}_g,$ where $\mathcal{H}_g$ is the moduli space of smooth hyperelliptic curves of genus $g$: for $g=3$ any non-hyperelliptic curve admits infinitely many pencils of degree 3, while, for $g=4,$ any non-hyperelliptic curve admits either one or two of them. On the other hand, when $g\geq5$ a non-hyperelliptic curve is not necessarily trigonal and, in particular, $\mathcal{T}_5$ represents the first case where its cohomology cannot be automatically deduced from that of $\mathcal{H}_5$ and $\mathcal{M}_5.$ In fact, $\mathcal{M}_5$ can be decomposed into the disjoint union of the moduli spaces of hyperelliptic curves $\mathcal{H}_5$, of trigonal curves $\mathcal{T}_5,$ and the one parametrizing curves that are the complete intersection of three linearly independent smooth quadric hypersurfaces, which will be denoted by $\mathcal{Q}_5$. Therefore the rational cohomology of $\mathcal{T}_5$ represents an advance not only in the understanding of that of $\mathcal{M}_5,$ which is unknown, but hopefully also of the cohomology of $\mathcal{T}_g,$ for any $g\geq 5$.\\
	What is known about $\mathcal{T}_g\cup\mathcal{H}_g$ until now is due to Stankova, who computed in \cite{Sta} the rational Picard group of its closure $\overline{\mathcal{T}}_g\subseteq\overline{\mathcal{M}}_g,$ i.e. the compactification of $\mathcal{T}_g\cup\mathcal{H}_g$ by admissible covers; and to Bolognesi and Vistoli, who computed in \cite{bolognesi2012stacks} the integral Picard group of $\mathcal{T}_g\cup\mathcal{H}_g$. Later, Patel and Vakil computed its whole rational Chow ring $A_{\mathbf{Q}}^*(\mathcal{T}_g\cup\mathcal{H}_g)$, \cite{PV}.

	 More recently, for $g=5,$ Wennink, in \cite{Wen}, counted the number of points of $\mathcal{T}_5$ over a finite field $\mathbf{F}_q$ with $q$ points:
	
	$$|\mathcal{T}_5(\mathbf{F}_q)|=q^{11}+q^{10}-q^8+1.$$
	By standard comparison theorems, this determines the Euler characteristic of $\mathcal{T}_5$ in $K_0(\mathsf{HS}_{\mathbf{Q}})$, the Grothendieck group of rational Hodge structures. We will refine Wennink's result and compute the rational cohomology of $\mathcal{T}_5$ with its mixed Hodge structures.
	\begin{teo}
	The rational cohomology of $\mathcal{T}_5$ is 
	$$H^{i}(\mathcal{T}_5;\mathbf{Q})=\begin{cases}
	\mathbf{Q}, & i=0;\\
	\mathbf{Q}(-1), & i=2;\\
	\mathbf{Q}(-3), & i=5;\\
	\mathbf{Q}(-11), & i=12;\\
	0, & \text{otherwise};
	\end{cases}$$
	where $\mathbf{Q}(-k)$ denotes the Hodge structure of Tate of weight $2k.$
	\end{teo}
		
	The whole rational cohomology of $\mathcal{T}_5$ can also be expressed in terms of its Hodge-Grothendieck polynomial, defined as
	\begin{equation}
		P(T_{\bullet};\mathbf{Q}):=\sum_{i\in\mathbf{Z}}\left[T_i\right]t^i\in K_0(\mathsf{HS}_{\mathbf{Q}})\left[t\right],\label{HG}
	\end{equation} 
	for any $\mathbf{Q}\mbox{-}$graded vector space of mixed Hodge structures $T_{\bullet}$.\\ 
	By Theorem 1.1, then 
	$$P(\mathcal{T}_5;\mathbf{Q})=\mathbf{L}^{11}t^{12}+\mathbf{L}^3t^5+\mathbf{L}t^2+1,$$
	where $\mathbf{L}$ denotes the class of the Tate Hodge structure $\mathbf{Q}(-1).$\\
	Moreover, since the moduli space $\mathcal{H}_g,$ $g\geq2,$ has always the rational cohomology of a point, we can also prove the following 
	\begin{cor}
	The rational cohomology of ${\mathcal{T}}_5\cup\mathcal{H}_5$ is 
	$$H^{i}({\mathcal{T}}_5\cup\mathcal{H}_5;\mathbf{Q})=\begin{cases}
	\mathbf{Q}, & i=0;\\
	\mathbf{Q}(-1), & i=2;\\
	\mathbf{Q}(-2), & i=4;\\ 
	\mathbf{Q}(-3), & i=5;\\
	\mathbf{Q}(-11), & i=12;\\
	0, & \text{otherwise}.
	\end{cases}$$
	\end{cor}

	Theorem 1.1 and Corollary 1.2 are consistent with the known results about their cohomology. In particular, the maximal weight class can be identified with the top weight cohomology class of $\mathcal{M}_5,$ described by Chan, Galatius and Payne in \cite{chan2018tropical} and \cite{chan2019topology}. They proved indeed that the cohomology $H^{4g-6}(\mathcal{M}_g;\mathbf{Q})$ is non zero for g=5 and that, by studying the dual complex of the boundary divisor in $\overline{\mathcal{M}}_g,$ the top graded piece on the cohomology of $\mathcal{M}_5$ is such that
	$$
	\text{dim}\text{Gr}^W_{6g-6}H^i(\mathcal{M}_5;\mathbf{Q})=
	\begin{cases}1,&i=14;\\
	0,&i\neq 14.
	\end{cases}
	$$
	The stratification of $\mathcal{M}_5$ as union of affine varieties
	$${\mathcal{T}}_5\cup\mathcal{H}_5\overset{\text{closed}}{\hookrightarrow}\mathcal{M}_5\overset{\text{open}}\hookleftarrow \mathcal{Q}_5$$
	induces a Gysin exact sequence in Borel-Moore homology
	$$\dots\rightarrow\bar{H}_k({\mathcal{T}}_5\cup\mathcal{H}_5;\mathbf{Q})\rightarrow \bar{H}_k({\mathcal{M}}_5;\mathbf{Q})\rightarrow \bar{H}_{k-1}({\mathcal{Q}}_5;\mathbf{Q})\rightarrow\dots$$
	Since $\mathcal{Q}_5$ is affine, $\bar{H}_{10}(\mathcal{M}_5;\mathbf{Q})=\bar{H}_{10}(\mathcal{T}_5;\mathbf{Q}),$ and Poincar\'e duality gives   $H^{14}(\mathcal{M}_5;\mathbf{Q})=H^{12}(\mathcal{T}_5;\mathbf{Q}).$\\ 
	
	We also notice that, for lower degrees, the rational cohomology of $\mathcal{T}_5\cup\mathcal{H}_5$ coincides with that of $\mathcal{M}_4,$ so one may wonder if this also happens for any $g\geq5$.\\

	The proof of Theorem 1.1 relies on Gorinov-Vassiliev's method. This method was first used by Vassiliev in \cite{Vart}, then modified by Gorinov in \cite{gorinov2005real} and finally by Tommasi in \cite{Tom}. We will use Tommasi's version of the method.\\
	The work is organized as follows. In Section 2 we recall some definitions and classical results. Then we present Gorinov-Vassiliev's method in Section 3 and we apply it in Section 4 in order to prove the main theorem.\\
	\\
	\textbf{Acknowledgements.} I am sincerely grateful to my advisor, Orsola Tommasi, for guiding me throughout this work and for her precious help.

	\section{Preliminaries}
	We recall that any trigonal curve of genus $g$ may be embedded in a Hirzebruch surface $\mathbb{F}_n$ as a divisor of class
	$$C\sim 3E+\frac{g+3n+2}{2}F,$$
	where $E$ is the exceptional divisor and $F$ is any fiber of the ruling $\mathbb{F}_n\rightarrow\mathbf{P}^1.$\\
	The integer $n$ is called the \textit{Maroni invariant} \cite{Mar} of $C$ and it has to satisfy
	$$g\equiv n \text{ mod } 2,\qquad\text{and}\qquad0\leq n\leq\frac{g+2}{3}.$$
	The moduli space of trigonal curves of genus $g$ has then a natural stratification given by the Maroni invariant:
	$$\mathcal{N}_{\lfloor{\frac{g+2}{3}}\rfloor}\subset\dots\subset\mathcal{N}_0=\mathcal{T}_g,$$
	where $\mathcal{N}_n$ is the closed subscheme of trigonal curves having Maroni invariant bigger or equal than $n.$\\
	Thus, for $g=5$, this stratification consists of only one stratum and any trigonal curve lies on the Hirzebruch surface $\mathbb{F}_1$, which is the blow up of the projective plane at one point, as an element of the linear system $|3E+5F|$.
	Moreover, one can show that there is a one-to-one correspondence between trigonal curves of genus 5 and projective plane quintics having exactly one ordinary node or cusp. In fact, given a trigonal curve $C$, and hence a $g^1_3$, one can show that the linear system $|K-D|,$ where $D\in g^1_3$ and $K$ is the canonical divisor, is a base point free $g_5^2.$ This defines a morphism $C\rightarrow \mathbf{P}^2$ such that the image has degree 5, with precisely one singularity of delta invariant 1. Conversely, a plane projective curve of degree 5 with one singularity that is a node or a cusp has arithmetic genus 5, and each line through the singular point meets the curve in other 3 points, counting multiplicity, defining a $g_3^1.$
	
	\subsection{Configuration spaces}
	We will work mostly with configuration spaces that we will define in this section. We also recall some results about their Borel-Moore homology with twisted coefficients that we will use repeatedly.
	\begin{defin}
		Let $Z$ be a topological space,
		$$F(Z,k):=Z^k\backslash\bigcup_{1\leq i<j\leq k}\{(z_1,\dots,z_k)\in Z^k|z_i=z_j\}$$
		is the space of ordered configurations of $k$ points in $Z.$ The quotient by the natural action of the symmetric group $\mathfrak{S}_k$ is denoted by $B(Z,k)$ and it is the space of unordered configurations of $k$ points in $Z.$\\
		For any subspace $Y\subseteq B(Z,k),$ the local system $\pm\mathbf{Q}$ over $Y$ is the one locally isomorphic to $\mathbf{Q},$ that changes its sign under any loop defining an odd permutation in a configuration from $Y.$ We will denote by $\bar{H}_{\bullet}(Y;\pm\mathbf{Q})$ the \emph{ Borel-Moore homology }of $Y$ with twisted coefficients, or the \emph{twisted Borel-Moore homology} of $Y,$ and by $\bar{P}(Y;\pm\mathbf{Q})$ its Hodge-Grothendieck polynomial, defined as in \eqref{HG}.
	\end{defin}
	
	When $P\in\mathbf{P}^2$ and $Z=\mathbf{P}^2\backslash\{P\},$ we will consider the subspaces $\tilde{F}(Z,k)\subseteq F(Z,k)$ and $\tilde{B}(Z,k)\subseteq B(Z,k)$ consisting of \emph{generic} configurations, i.e. configurations of points that are in general position, such that no two points lie on the same line through $P$.\\

	We also list some Lemmas that will be fundamental in our next computations.
	\begin{lem}[\cite{Vart}]\begin {itemize}
		\item[a.] $\bar{H}_{\bullet}(B(\mathbf{C}^N,k);\pm\mathbf{Q})$ is trivial for any $N\geq1$, $k\geq2.$
		\item[b.] $\bar{H}_{\bullet}(B(\mathbf{P}^N,k);\pm\mathbf{Q})=H_{\bullet-k(k-1)}(G(k,\mathbf{C}^{N+1});\mathbf{Q})$, where $G(k,\mathbf{C}^{N+1})$ is the Grassmann manifold of $k\mbox{-}$dimensional subspaces in $\mathbf{C}^{N+1}$. In particular the group $\bar{H}_{\bullet}(B(\mathbf{P}^N,k);\pm\mathbf{Q})$ is trivial if $k>N+1.$
	\end{itemize}
\end{lem}

\begin{lem}[\cite{Tom}]
	The Hodge-Grothendieck polynomial of $\bar{H}_{\bullet}(B(\mathbf{C}^*,k);\pm\mathbf{Q})$ is $t^k+\mathbf{L}^{-1}t^{k+1}$ for any $k\geq1.$ If we consider the action of $\mathfrak{S}_2$ on $\mathbf{C}^*$ induced by $\tau\mapsto\frac{1}{\tau},$ we have that the Borel-Moore homology classes of even degree are invariant and those of odd degree are anti-invariant.
\end{lem}

\begin{lem}
	The Hodge-Grothendieck polynomial of $\bar{H}_{\bullet}(B(\mathbf{P}^2\backslash\{P\},2);\pm\mathbf{Q})$ is $\mathbf{L}^{-3}t^6$.\\
	$\bar{H}_{\bullet}(B(\mathbf{P}^2\backslash\{P\},k);\pm\mathbf{Q})$ is trivial for $k\geq3$.
\end{lem}\begin{proof}$\mathbf{P}^2\backslash\{P\}$ can be decomposed into the disjoint union of spaces $S_1,S_2$, isomorphic respectively to $\mathbf{C}^2$ and $\mathbf{C}.$ Then, to any configuration of points in $B(\mathbf{P}^2\backslash\{P\},k)$ we can associate an ordered partition $(a_1,a_2)$, where $a_i$ is the number of points contained in $S_i.$ We consider each possible partition of $k$ as defining a stratum in $B(\mathbf{P}^2\backslash\{P\},k)$, and order each stata by lexicographic order of the index of partition. All strata with any $a_i\geq2$ have no twisted Borel-Moore homology by Lemma 2.1.$a$, so the second part of the Lemma is proved. When $k=2,$ the only admissible partition is $(1,1)$ that is a stratum isomorphic to $\mathbf{C}^3,$ hence it has twisted Borel-Moore homology $\mathbf{Q}(3)$ in degree 6 and trivial homology in all other degrees.
\end{proof}

\begin{lem}
	The Hodge-Grothendieck polynomial of $\bar{H}_{\bullet}(\tilde{F}(\mathbf{P}^2\backslash\{P\},2);\mathbf{Q})$ is 
	$\mathbf{L}^{-4}t^{8}+\mathbf{L}^{-3}t^{6}.$
\end{lem}\begin{proof}
Using the definition, one can compute the Hodge-Grothendieck polynomial of $\bar{H}_{\bullet}({F}(\mathbf{P}^2\backslash\{P\},2);\mathbf{Q})$ that is $\mathbf{L}^{-4}t^8+2\mathbf{L}^{-3}t^6+\mathbf{L}^{-1}t^3.$\\
The space $F(\mathbf{P}^2\backslash\{P\},2)\backslash\tilde{F}(\mathbf{P}^2\backslash\{P\},2)$ consists of couple of points lying on the same line through $P:$ it is fibered over $(P)^{\vee}\cong\mathbf{P}^1$ with fiber equal to $F(\mathbf{C},2).$ Therefore the Borel-Moore homology of $F(\mathbf{P}^2\backslash\{P\},2)\backslash\tilde{F}(\mathbf{P}^2\backslash\{P\},2)$ is $\mathbf{Q}(3)$ in degree 6 and $\mathbf{Q}(1)$ in degree 3, and by considering the long exact sequence induced by the short exact sequence
$$0\rightarrow\tilde{F}(\mathbf{P}^2\backslash\{P\},2)\rightarrow{F}(\mathbf{P}^2\backslash\{P\},2)\rightarrow F(\mathbf{P}^2\backslash\{P\},2)\backslash\tilde{F}(\mathbf{P}^2\backslash\{P\},2)\rightarrow 0$$
the lemma is proved.
\end{proof}

\begin{lem}
	The Hodge-Grothendieck polynomial of $\bar{H}_{\bullet}(\tilde{F}(\mathbf{P}^2\backslash\{P\},3);\mathbf{Q})$ is $\mathbf{L}^{-6}t^{12}+\mathbf{L}^{-5}t^{11}+\mathbf{L}^{-4}t^9+\mathbf{L}^{-3}t^{8}$.
\end{lem}\begin{proof}
Similarly to the previous proof, we compute first the Hodge-Grothendieck polynomial of $\bar{H}_{\bullet}({F}(\mathbf{P}^2\backslash\{P\},3);\mathbf{Q}),$ that is $\mathbf{L}^{-6}t^{12}+3\mathbf{L}^{-5}t^{10}+5\mathbf{L}^{-3}t^7+\mathbf{L}^{-2}t^{5}+2\mathbf{L}^{-1}t^{4},$ and we then consider the complement of $\tilde{F}(\mathbf{P}^2\backslash\{P\},3)$ in ${F}(\mathbf{P}^2\backslash\{P\},3).$ This space is the union of 3 pieces: the space of triples lying on the same line not passing through $P,$ the space of triples lying on the same line through $P,$ and the space consisting of triples where exactly 2 points lie on the same line through $P.$ The Hodge-Grothendieck polynomials of their Borel-Moore homology are respectively $\mathbf{L}^{-5}t^{10}+\mathbf{L}^{-3}t^7$, $\mathbf{L}^{-4}t^8+2\mathbf{L}^{-3}t^7+\mathbf{L}^{-2}t^5+2\mathbf{L}^{-1}t^4$ and $3\mathbf{L}^{-5}t^{10}+3\mathbf{L}^{-3}t^7.$\\
By computing first the Borel-Moore homology of the union of these three spaces and then considering the long exact sequence induced by 
$$0\rightarrow\tilde{F}(\mathbf{P}^2\backslash\{P\},3)\rightarrow{F}(\mathbf{P}^2\backslash\{P\},3)\rightarrow F(\mathbf{P}^2\backslash\{P\},3)\backslash\tilde{F}(\mathbf{P}^2\backslash\{P\},3)\rightarrow 0$$ we finally get that of $\tilde{F}(\mathbf{P}^2\backslash\{P\},3).$
\end{proof}

\begin{lem}
	There are isomorphisms
	$$\bar{H}_{\bullet}(\tilde{B}(\mathbf{P}^2\backslash\{P\},2);\pm\mathbf{Q})\xrightarrow{\sim}\bar{H}_{\bullet}({B}(\mathbf{P}^2\backslash\{P\},2);\pm\mathbf{Q})$$
	$$\bar{H}_{\bullet}(\tilde{B}(\mathbf{P}^2\backslash\{P\},3);\pm\mathbf{Q})\xrightarrow{\sim}\bar{H}_{\bullet}({B}(\mathbf{P}^2\backslash\{P\},3);\pm\mathbf{Q})$$
	induced by the natural inclusions.
\end{lem}
\begin{proof}
	As we noticed before, the spaces $B(\mathbf{P}^2\backslash\{P\},2)\backslash\tilde{B}(\mathbf{P}^2\backslash\{P\},2)$ and $B(\mathbf{P}^2\backslash\{P\},3)\backslash\tilde{B}(\mathbf{P}^2\backslash\{P\},3)$ are union of fiber spaces with fibers $B(\mathbf{C},2)$, $B(\mathbf{C},3)$ or $B(\mathbf{P}^1,3)$ and by Lemma 2.1 all these fibers have trivial twisted Borel-Moore homology.
\end{proof}

\begin{lem}[\cite{gorinov2005real}]\label{cor3.5}
	Let $p: N'\rightarrow N$ be a finite sheeted covering of manifolds, and let $\mathcal{L}$ be a local system of coefficients on $N'$. Then $\bar{H}_{\bullet}(N',\mathcal{L})=\bar{H}_{\bullet}(N,p(\mathcal{L}))$, where $p(\mathcal{L})$ denotes the direct image of the system $\mathcal{L}.$
\end{lem}

\section{Gorinov-Vassiliev's method}
From the previous section, we recall that trigonal curves of genus 5 can be either described as elements of the linear system $|3E+5F|$ in the surface $Z=\mathbb{F}_1,$ or as projective plane quintics having exactly a node or a cusp. Let $P\in\mathbf{P}^2$ be this singular point.\\
We define $V$ to be the vector space of global sections of $\mathcal{O}_Z(3E+5F)$. Then, $V$ is isomorphic to the vector space of polynomials defining plane curves of degree 5 having at least a singular point at $P.$\\
The space $\mathcal{T}_5$ can be represented as the quotient of an open subset of $V$ by the action of the automorphism group of $Z$. The open subset we are interested in, that we will denote by $X$, is the subset in $V$ defining smooth curves on $Z,$ i.e. curves meeting the exceptional divisor exactly twice.
The complement of $X$ in $V$ will be denoted by $\Sigma$ and it is called the \textit{discriminant hypersurface.} To compute the cohomology of $X$ we will use the Gorinov-Vassiliev's method, that consists of computing the Borel-Moore homology of the discriminant, thanks to the Alexander duality:
$$\tilde{H}^{\bullet}(V\backslash\Sigma;\mathbf{Q})\cong H^{\bullet+1}(V,V\backslash\Sigma;\mathbf{Q})\cong\bar{H}_{2(N-1-\bullet)}(\Sigma;\mathbf{Q})(-N),$$
where $N$ is the complex dimension of $V.$ \\
Following \cite[Section 2.1]{Tom}, the Borel-Moore homology of $\Sigma$ is obtained by constracting a simplicial resolution $\sigma$ of $\Sigma$, starting from a collection $(X_i)_{i\in I }$ of families of configurations in $Z,$ satisfying the axioms that are listed in \cite[List 2.3]{Tom}. Then we define an increasing filtration $(\text{Fil}_i)_{i\in I}$ of $\sigma$, whose associated spectral sequence is proved to converge to the Borel-Moore homology of the discriminant:
\begin{prop}[\cite{Tom}]
	The filtration $\text{Fil}_i$ defines a spectral sequence that converges to the Borel-Moore homology of $\Sigma$, whose $E^1_{p,q}\mbox{-}$term is isomorphic to $\bar{H}_{p+q}(F_p;\mathbf{Q}),$
	where $F_i:=\text{Fil}_i\backslash \text{Fil}_{i-1}$.
\end{prop}

In particular, if $V$ is the vector space of global sections of $\mathcal{O}_Z(3E+5F)$ or equivalently the vector space of polynomial equations of degree 5 whose zero locus are plane projective curves having at least a node or a cusp at a fixed point $P$, the dimension of $V$ is 18.\footnotemark\footnotetext{The dimension of the vector space of polynomials defining plane quintics having at least a singular point is in fact $21-3=18$.}
To compute the dimension, we consider a general plane quintic having at least a node or a cusp at the point that we will blow up: $P=\left[1,0,0\right]$, that is a curve defined by a polynomial $f\in\mathbf{C}\left[x_0,x_1,x_2\right]$ having degree $\leq3$ with respect to the variable $x_0.$ For any such curve we can consider a projection with center $P$: fix a line $l$ not passing through $P$, for example $l:=\{\left[0,y_1,y_2\right]\}$, and take the map sending all points of the curve distinct from $P$ to the point of intersection between the line connecting the point to $P$ and $l$. 
\\
The preimage of any point through this map is given by points of the curve on the same line through $P$, which has parametric equation
$$r:\begin{cases} x_0=t_0\\ x_1=t_1y_1\\x_2=t_1y_2
\end{cases},\qquad \left[t_0,t_1\right]\in\mathbf{P}^1.$$
Hence, because any line through $P$ corresponds to a line of the ruling in the blow up, any curve we want to consider can be embedded in the blow up via the mapping
$$\left[t_0,t_1y_1,t_1y_2\right]\hookrightarrow\left[t_0,t_1y_1,t_1y_2\right]\times\left[y_1,y_2\right]$$ such that $f(t_0,t_1y_1,t_1y_2)=0,$ i.e. it has equation
$$t_1^2(t_0^3g_2(y_1,y_2)+t_0^2t_1g_3(y_1,y_2)+t_0t_1^2g_4(y_1,y_2)+t_1^3g_5(y_1,y_2)),$$ where each $g_i$ is a homogeneous polynomial of degree $i.$ Counting the number of parameters will indeed give 18.\\
The automorphism group of the homogeneous coordinate ring of $Bl_{P}\mathbf{P}^2$ is the set of automorphisms of the graded ring $\mathbf{C}\left[x_0,x_1,x_2\right]$ that fix the point that is blown up, i.e.
$$G=\left\{\begin{bmatrix} *&*&* \\
0&*&* \\ 0&*&*\end{bmatrix}\in GL(3,\mathbf{C})\right\} \supset \mathbf{C}^*\times GL(2,\mathbf{C}).$$
Note that ignoring the second and the third term in the first row of each of the matrices in $G$ means contracting the vector space $\mathbf{C}\left[x_1,x_2\right]_1\cong\mathbf{C}^2$ to a point, therefore $G$ is homotopy equivalent to $\mathbf{C}^*\times GL(2,\mathbf{C}).$\\
Note also that $G$ contains the normal unipotent subgroup
$$\left\{\begin{bmatrix}1&*&* \\
0&1&* \\ 0&0&1\end{bmatrix}\in GL(3,\mathbf{C})\right\}$$ hence it is not reductive and we cannot construct our moduli space as a quotient by $G.$\\
However, we can consider its reductive part $\mathbf{C}^*\times GL(2,\mathbf{C}),$ construct the GIT quotient $X/(\mathbf{C}^*\times GL(2,\mathbf{C}))$ and compute its cohomology instead: $X/(\mathbf{C}^*\times GL(2,\mathbf{C})$ is the space of isomorphism classes of triples $(C,L,H),$ where $C$ is a curve of genus $g,$ $L$ is the linear system defining its trigonal structure and $H$ is the hyperplane section corresponding to the line $l$ not meeting $P$ that we defined before, hence this is a $\mathbf{C}^2\mbox{-}$bundle over $\mathcal{T}_5$, in the orbifold sense, and therefore they have same rational cohomology.\\

In particular, we will first consider the reductive subgroup $\{1\}\times GL(2,\mathbf{C})\subset\mathbf{C}^*\times GL(2,\mathbf{C})$ and its GIT quotient $X/GL(2,\mathbf{C})$. Then we will compute its cohomology by using a generalized version of Leray-Hirsch theorem and we will consider the orbifold $\mathbf{C}^*\mbox{-}$bundle
$$X/GL(2,\mathbf{C})\xrightarrow{\mathbf{C}^*} X/(\mathbf{C}^*\times GL(2,\mathbf{C})).$$
and deduce the cohomology of the base space from the spectral sequence associated to this bundle.

\subsection{Generalized Leray-Hirsch theorem}
We want to prove that there exists an isomorphism of graded $\mathbf{Q}\mbox{-}$vector spaces with mixed Hodge structures
$$H^{\bullet}(X/GL(2,\mathbf{C});\mathbf{Q})\otimes H^{\bullet}(GL(2,\mathbf{C});\mathbf{Q})\cong H^{\bullet}(X;\mathbf{Q}).$$
By \cite[Th.2]{PS} it suffices to prove the surjectivity of the orbit map on cohomology
$$\rho^*: \bar{H}_{2\operatorname{dim}V-i-1}(\Sigma;\mathbf{Q})\cong H^i(X;\mathbf{Q})\rightarrow H^i(GL(2,\mathbf{C});\mathbf{Q})\cong \bar{H}_{2\operatorname{dim}M-i-1}(D;\mathbf{Q}),$$
where $M$ denotes the space of $2\times 2$ matrices and $D$ the discriminant of $GL(2,\mathbf{C})$ in $M.$\\

We know that the cohomology of $GL(2,\mathbf{C})$ has generators in degrees $i=1,3,$ and the generators of $\bar{H}_{\bullet}(D;\mathbf{Q})$ are $\left[D\right]\in\bar{H}_{6}(D;\mathbf{Q})$ and $\left[R\right]\in \bar{H}_{4}(D;\mathbf{Q}),$ where we can assume $R$ to be the subvariety of matrices with only zeros in the first column.\\
Moreover, from the spectral sequence that will be exhibited in Section 4.7, $\bar{H}_{34}(\Sigma;\mathbf{Q})=\langle\left[\Sigma\right]\rangle,$ and $\bar{H}_{32}(\Sigma;\mathbf{Q})=\langle\left[\Sigma_1\right],\left[\Sigma_2\right]\rangle$, where $\Sigma_1$ is the subspace in $V$ of polynomials defining curves having a singular point on $E$, and $\Sigma_2$ is the subspace in $V$ of polynomials defining a singularity on a fixed line of the ruling $L.$\\

Let's consider the extension of the orbit map
$$
D\rightarrow \Sigma$$
and the image of an element in $R:$

$$A=\begin{pmatrix}

0&b\\0&d
\end{pmatrix}\mapsto A\cdot f(x_0,x_1,x_2)=f(x_0,bx_2,dx_2)=\alpha(b,c)x_2^2 h_3(x_0,x_2),$$
where $\alpha(b,c)$ is some constant and $h_3$ is the product of $3$ lines through the point $\left[0,1,0\right].$ So, elements of $R$ are mapped to polynomials whose zero loci are the union of a double line of the ruling of fixed equation $y_2=0$ and three lines through a point of the ruling $(\left[t_0,t_1,0\right],[1,0]).$\\
Similarly, elements in $D$ are mapped to curves which are the union of any double line of the ruling and three lines through a point of that ruling.\\
Hence we can deduce that $\rho^*(\left[\Sigma\right])$ is a non-zero multiple of $\left[D\right],$ while the preimage of $\left[R\right]$ through $\rho^*$ must be a non-trivial linear combination of $\left[\Sigma_1\right],\left[\Sigma_2\right],$ proving the surjectivity of the map in cohomology.

\section{Application of the method}

In this section we want to apply the method introduced in the previous section. First of all, we produce a list of all the possible configurations of singularities of genus 5 curves in $Bl_P\mathbf{P}^2$ meeting the exceptional divisor $E$ at least twice.\\
To do so, we recall that we are considering curves in $\mathbb{F}_1$ which are elements of the linear system $|3E+5F|.$
Since all singularities are obtained as degenerations of nodes, we will first consider only such singularities.\\
Assume that the curve is irreducible: by computing the arithmetic genus we get an upper bound for the number of singularities. For instance, by the \emph{genus formula}, we have that
$$g(3E+5F)=1+\frac{1}{2}((3E+5F)^2+(3E+5F)\cdot K)=5,$$
where $K$ is the canonical divisor on $\mathbb{F}_1,$ so we can have at most 5 ordinary double points.\\
Then, we will consider all the possible ways in which the curve can be reduced. Here we will have to take into account not only the singularities of each irreducible component, but also all the intersections between them.\\
Finally we will consider all the possible degenerations of the singularities obtained in this way (points can be on the exceptional divisor or points in general position can become collinear, etc...) and all the subsets of finite configurations. \\
For any configuration of singularities, the elements in $V$ which are singular at least at that configuration form a vector space and we can compute its codimension in $V$, that will be written in brackets. By ordering all the configurations obtained by increasing codimension, and then by increasing number of points, we will get a list of configurations indexed by $(j)$, and by defining $X_j$ as the space of configurations of type $(j)$ we will get a sequence of families of configurations that will satisfy conditions 1-7 in \cite[List 2.3]{Tom}.\\
We report only a shorter version of this list, omitting for instance all configurations containing rational curves since they will give no contribution to the Borel-Moore homology of the discriminant by \cite[Lemmas 2.19, 2.20]{Tom}, and combining similar configurations that will also give no contribution.
\\
In the following, we denote a configuration of points by \emph{general} if it is a configuration of points in general position, where no point is contained in $E$ and no two points lie in the same line of the ruling. 
Note that we consider a single point to be `general' also if it belongs 
to a line of the ruling passing through a point in $E$ contained in the same configuration (this latter configuration is contained in the one where the point is in $P^2\backslash\{P\}$). 
We will also use the following notation:\\

\begin{tabular}{l l}
\emph{line of the ruling} & it is an element in $|F|,$\\
& i.e. the strict transform of a line in $\mathbf{P}^2$ passing through $P;$\\
\emph{line} & it is an element in $|E+F|,$\\
& i.e. the strict transform of a line in $\mathbf{P}^2$ not passing through $P;$\\
\emph{conic} $C_P$& it is an element in $|E+2F|,$ \\
&i.e. the strict transform of a conic in $\mathbf{P}^2$ passing through $P;$\\
\emph{conic} $C$ & it is an element in $|2E+2F|,$\\
& i.e. the strict transform of a line in $\mathbf{P}^2$ not passing through $P.$
\end{tabular}

\begin{enumerate}
\item A point on the exceptional divisor $E;$ $\left[3\right]$
\item A general point; $\left[3\right]$
\item Two points on $E;$ $\left[5\right]$
\item Two (or three) points on a line of the ruling;  $\left[6\; (7)\right]$ 
\item A point on $E$ + a general point; $\left[6\right]$
\item Two general points; $\left[6\right]$
\item Three points or more points on $E;$ $\left[6\right]$
\item Two points on $E$ + a general point; $\left[8\right]$
\item A point \emph{(that can be either on $E$, or general)} + two (or three) points on a line of the ruling; $\left[9\;(10)\right]$ 


\item A point on $E$ + two general points;$\left[9\right]$
\item Three (or four) points on a line $L$; $\left[9\;(10)\right]$ 
\item Three (or resp. four, five) general points; $\left[9\;(12,15)\right]$ 
\item A point on $E$ + two (or three) points on a line of the ruling $F$ + the point of intersection between $F$ and $E;$ $\left[9\;(10)\right]$ 

\item Three points on $E$ + a general point; $\left[9\right]$
\item Two points on $E$ + two (or three) points on a line of the ruling; $\left[10\;(11)\right]$ 

\item Two points on $E$ + two (or three) points on a line of the ruling $F$ + the point of intersection between $F$ and $E$; $\left[10\;(11)\right]$


\item Two points on $E$ + two general points; $\left[11\right]$
\item Three points on $E$ + two (or three) points on a line of the ruling; $\left[11\;(12)\right]$ 
\item Three points on $E$ + two (or three) points on a line of the ruling $F$ + the point of intersection between $F$ and $E$; $\left[11\;(12)\right]$

\item Two points on each of two lines of the ruling (or resp. two points on a ruling and three points on the other one, or three points on each of the two rulings); $\left[12\;(13,14)\right]$ 
\item Two general points + two (or three) points on a line of the ruling; $\left[12\;(13)\right]$
\item  A point \emph{(that can be either on $E$, or general)} + three (or four) points on a line $L$; $\left[12\;(13)\right]$ 
\item A point on $E$ + three (or four) general points; $\left[12\;(15)\right]$ 
\item A point on $E$ + two (or three) points on a line of the ruling $F $+ a general point; $\left[12\;(13)\right]$
\item A point on $E$ + two (or three) points on a line of the ruling $F $+ a general point + the point of intersection between $F$ and $E$; $\left[12\;(13)\right]$ 
\item Three points on $E$ + two general points; $\left[12\right]$

\item Two points on $E$ + three (or four) points on a line $L$; $\left[14\;(15)\right]$ 
\item Two points on $E$ + three general points; $\left[14\right]$
\item Two points on a line of the ruling $F$ + three points on a line $L$; $\left[14\right]$
\item Two points on a line of the ruling $F$ + three points on a line $L$ + the point of intersection between $F$ and $L$; $\left[14\right]$
\item Five (or six) points on a conic $C_P$ ($C$); $\left[14\;(17)\right]$ 
\item Two points on each of two lines of the ruling $F_1, F_2$ + the intersection points with $E$ + a point on $E$; $\left[14\right]$
\item Two points on each of two lines + the point of intersection; $\left[15\right]$
\item Two points on each of two lines of the ruling (or two points on a ruling and three on the other one) + a general point; $\left[15\;(16)\right]$ 
\item A point on $E$ + two general points + two or more points on a line of the ruling; $\left[15\right]$
\item Three general points + two (or three) points on a line of the ruling; $\left[15\;(16)\right]$ 
\item A point on $E$ + a general point + three (or four) points on a line; $\left[15\;(16)\right]$
\item Two general points + three (or four) points on a line; $\left[15\;(16)\right]$ 
\item Three points on $E$ + three (or four) points on a line; $\left[15\;(16)\right]$ 
\item Three points on $E$ + three general points; $\left[15\right]$

\item Two points on $E$ + two general points + two points on a line of the ruling $F$+ the point of intersection between $E$ and $F$; $\left[16\right]$
\item Five points on a conic $C_P$ + a general point; $\left[17\right]$
\item Three points on $E$ + four points on a conic $C_P$; $\left[17\right]$

\item Two points on a ruling $F$ + three points on a line $L$ + the intersection point between $F$ and $L$ + a general point; $\left[17\right]$
\item Three points on each of two rulings + a general point; $\left[17\right]$
\item Three points on each of two lines + the point of intersection; $\left[17\right]$

\item 7 points: three points of intersection between two conics $C_P$ and $C_P'$, one of which on $E$ + four points of intersection with a line; $\left[17\right]$
\item 7 points: three points of intersection between two conics $C_P$ and $C_P'$, none of which on $E$ + four points of intersection with a line;  $\left[17\right]$

\item 7 points: four points of intersection between two conics $C,$ $C_P$+ three points of intersection with a line of the ruling; $\left[17\right]$
\item 8 points: a point on $E$ + two points on each of two rulings $F_1$, $F_2$ + the points of intersection between $F_1$ and $E$, and $F_2$ and $E$ + a general point;  $\left[17\right]$
\item 8 points: two points on $E$ + three points on a line $L$ + the intersection points of a line of the ruling $F$ with $E$ and $L$ + another point on $F$;  $\left[17\right]$
\item 8 points: three points of intersection of two conics $C_P$ and $C_P'$, each meeting a line of the ruling $F$ and $E$ at one point + the point of intersection between $F$ and $E$;  $\left[17\right]$

\item 8 points: two points of intersection between two lines of the ruling $F_1$, $F_2$ and a line $L$ + 6 points of intersection with a conic $C$ meeting each line at two distinct points; $\left[17\right]$
\item 8 points: three points of intersection between a line of the ruling $F$ and two lines $L_1$, $L_2$ + five points of intersection with a conic $C_P$ meeting each line twice and $F$ only once, outside $E$; $\left[17\right]$

\item 8 points: three points of intersection between a line of the ruling $F$ and two lines $L_1$, $L_2$ + five points of intersection with a conic $C_P$ meeting each line twice and $F$ at the intersection point with $E$;  $\left[17\right]$

\item 9 points: four points of intersection between $E$, two lines of the ruling $F_1$, $F_2$ and a line $L$ + five points of intersection with a conic $C_P$ meeting $L$ twice and $E,$ $F_1,$ $F_2$ once;  $\left[17\right]$

\item 9 points: three points of intersection between $E$ and three lines of the ruling $F_1$, $F_2$, $F_3$ + 6 points of intersection with a conic $C$ meeting each $F_i$ at two distinct points; $\left[17\right]$

\item 9 points: the points of intersection between two lines of the ruling and three general lines;  $\left[17\right]$
\item 10 points: the points of intersection between $E$, three lines of the ruling and two lines;  $\left[17\right]$

\item The whole $Bl_P{\mathbf{P}^2}$. $\left[18\right]$
\end{enumerate}

Since simplicial bundles are non orientable, we will consider the Borel- Moore homology with coefficients in the local system $\pm\mathbf{Q}.$\\
We also recall from Lemmas 2.1 and 2.3 that configurations with at least three points on a rational curve, configurations with at least two points on a rational curve minus a point, and configurations with at least three general points give no contribution.
Thus, among the first 41 configurations, only the following have non-trivial Borel-Moore homology: 
\begin{itemize}
\item[(A)] A point on $E$. $\left[3\right]$
\item[(B)] A general point. $\left[3\right]$
\item[(C)] Two points on $E$. $\left[5\right]$
\item[(D)] A point on $E$ + a general point. $\left[6\right]$
\item[(E)] Two general points. $\left[6\right]$
\item[(F)] One general point + two points on $E$.  $\left[8\right]$
\item[(G)] Two general points + one point on $E$.  $\left[9\right]$
\item[(H)] Two general points + two points on $E$.  $\left[11\right]$
\end{itemize}

We will also prove in Appendix A that there are only other four configurations having non-trivial Borel-Moore homology:
\begin{itemize}
	\item[(I)] 7 points: config. 47; $\left[17\right]$
	\item[(J)] 7 points: config. 48; $\left[17\right]$
	
	\item[(L)] 8 points: config. 55; $\left[17\right]$
	
	\item[(M)] Whole $Bl_{P}\mathbf{P}^2.$ $\left[18\right]$
\end{itemize}

Since we are studying singular configurations of curves that are equivalent to plane projective quintics having at least one singularity, we can deduce their Borel-Moore homology by considering their equivalent description in the projective plane, by fixing a point $P$ that is the one that, when blown-up, will give us the corresponding curve in $\mathbb{F}_1$.\\
Note that, the configuration spaces that we will consider in the following are empty unless they are defined as the singular loci of the plane quintics that will be described.

\subsection{Columns (A)-(H)}
By applying \cite[Th. 2.8]{gorinov2005real}, we get the following results.\\
The space $F_A$ is a $\mathbf{C}^{15}\mbox{-}$bundle over $X_A\cong\mathbf{P}^1.$ \\
The space $F_B$ is a $\mathbf{C}^{15}\mbox{-}$bundle over $X_B\cong\mathbf{P}^2\backslash\{pt\}.$\\
The space $F_C$ is a $\mathbf{C}^{13}\times\mathring{\Delta}_1\mbox{-}$bundle over $X_C\cong B(\mathbf{P}^1,2).$ \\
The space $F_D$ is a $\mathbf{C}^{12}\times\mathring{\Delta}_1\mbox{-}$bundle over $X_D\cong\mathbf{P}^1\times\mathbf{P}^2\backslash\{pt\}.$\\
The space $F_E$ is a $\mathbf{C}^{12}\times\mathring{\Delta}_1\mbox{-}$bundle over $X_E\cong B(\mathbf{P}^2\backslash\{pt\},2).$\\
The space $F_F$ is a $\mathbf{C}^{10}\times\mathring{\Delta}_2\mbox{-}$bundle over $X_F\cong\mathbf{P}^2\backslash\{pt\}\times B(\mathbf{P}^1,2).$\\
The space $F_G$ is a $\mathbf{C}^{9}\times\mathring{\Delta}_2\mbox{-}$bundle over $X_G\cong B(\mathbf{P}^2\backslash\{pt\},2)\times\mathbf{P}^1.$\\
The space $F_H$ is a $\mathbf{C}^{7}\times\mathring{\Delta}_3\mbox{-}$bundle over $X_H\cong B(\mathbf{P}^2\backslash\{pt\},2)\times B(\mathbf{P}^1,2).$\\

\subsection{Column (I)+(J)}
Each configuration in $X_I$ is defined as the singular loci of the blow up at $P$ of a plane quintic defined by two reduced conics tangent at $P$ and a line meeting the conics at four distinct points, as in Figure 1.\\

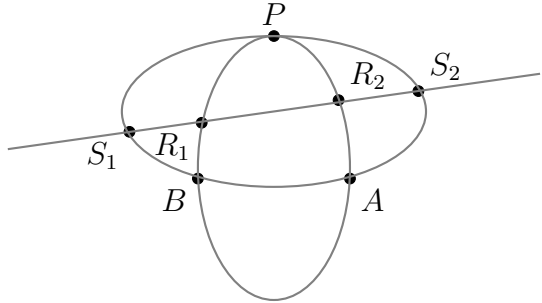
\begin{figure}[H]\centering
	\begin{tikzpicture}
	\filldraw[black] (13,1) circle (2pt) node[anchor=south] {$P$};
	\filldraw[black] (14,-0.89) circle (2pt) node[anchor=north west] {$A$};
	\filldraw[black] (12,-0.89) circle (2pt) node[anchor=north east] {$B$};
	\filldraw[black] (11.1,-0.27) circle (2pt) node[anchor=north east] {$S_1$};
	\filldraw[black] (12.05,-0.15) circle (2pt) node[anchor=north east] {$R_1$};
	\filldraw[black] (13.85,0.15) circle (2pt) node[anchor=south west] {$R_2$};
	\filldraw[black] (14.9,0.27) circle (2pt) node[anchor=south west] {$S_2$};
	\draw[gray, thick] (9.5,-0.5) -- (16.5,0.5);
	
	\draw[gray, thick] (13,0) ellipse (2 and 1);
	\draw[gray, thick] (13,-0.75) ellipse (1 and 1.75);
	\end{tikzpicture} 
	\caption{configuration of type (I)}
	\label{fig:config1}
\end{figure}

On the other hand, configurations of type (J) arise from blowing up $P,$ where $P$ is now one of the four points of intersection between two reduced conics that, together with a line not meeting the conics at any of the points of intersection, define the plane projective quintic curve in Figure 2.
\begin{figure}[H]\centering
	\begin{tikzpicture}
	
	\filldraw[black] (13.85,0.9) circle (2pt) node[anchor=south west] {$B$};
	\filldraw[black] (12.15,0.9) circle (2pt) node[anchor=south east] {$A$};
	\filldraw[black] (13.85,-0.9) circle (2pt) node[anchor=north west] {$C$};
	\filldraw[black] (12.15,-0.9) circle (2pt) node[anchor=north east] {$P$};
	\filldraw[black] (14,0.16) circle (2pt) node[anchor=south west] {$R_2$};
	\filldraw[black] (14.9,0.28) circle (2pt) node[anchor=south west] {$S_2$};
	\filldraw[black] (12,-0.16) circle (2pt) node[anchor=north east] {$R_1$};
	\filldraw[black] (11.1,-0.28) circle (2pt) node[anchor=north east] {$S_1$};
	\draw[gray, thick] (9.5,-0.5) -- (16.5,0.5);
	
	\draw[gray, thick] (13,0) ellipse (2 and 1);
	\draw[gray, thick] (13,0) ellipse (1 and 1.75);
	\end{tikzpicture}                                                      	\caption{configuration of type (J)}
	\label{fig:config2}
\end{figure}
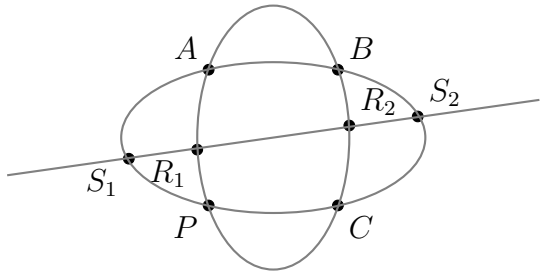

By noticing that the configuration space $X_I$ is contained in the closure of $X_J$ (by allowing one of the points $A,B,C$ to lie on the exceptional divisor $E$ of $\mathbb{F}_1$) we can consider a bigger configuration space containing both of them, which we will denote by $X_{I+J}.$\\
$X_{I+J}$ is the space of configurations of 7 singular points $A,B,C,R_1,R_2,S_1,S_2,$ such that $A,B,C$ are points that are in general position, also with respect to $P$, where only one of them is allowed to lie on $E$ coincide with $P,$ and $R_1,R_2,S_1,S_2$ are four distinct points of intersection between two distinct reduced conics passing through $A,B,C,P$ and a line $l$ not passing through any of these points.\\
We can fiber $X_{I+J}$ over the space parametrizing the points $A,B,C$ and the choice of the line:
$$X_{I+J}\rightarrow B:=\{(\{A,B,C\},l): A,B,C,l\text{ as in the description above}\}.$$

The fiber of this map, which we denote by $\mathcal{Z},$ will then be the space of pairs of conics passing through the four points and not tangential to the line. Note that $\mathcal{Z}$ is exactly the same fiber space considered in \cite[Section 4.2]{gorinov2005real} in Column 38.\\
As both conics have to satisfy 4 linear independent conditions (that consist either in the passage through 4 distinct points or 3 points plus the tangency condition), each of them is uniquely determined by a point on the line $l,$ and we will denote these points by $S_1,R_1$ as in the figures. 
Recall that there are exactly two conics in the pencil with base locus $A+B+C+P$ that are tangent to $l$. Let $T_1, T_2$ be the points of intersection between $l$ and the two tangent conics. Any other conic will meet $l$ at two distinct points. Exchanging these two intersection points defines an involution on $l\cong\mathbf{P}^1$ that fixes $T_1,T_2,$ and by choosing an appropriate coordinate system we can assume that $T_1=\left[1,0\right]$,  $T_2=\left[0,1\right]$ and this involution will be $\left[1,t\right]\mapsto\left[1,-t\right].$\\
Therefore we can set $S_1=\left[1,t\right]$, $S_2=\left[1,-t\right]$, $R_1=\left[1,s\right]$, $R_2=\left[1,-s\right]$ and the space $\mathcal{Z}$ parametrizing the two conics of the configuration will be a quotient of $$(t,s)\in\tilde{\mathcal{Z}}:=\mathbf{C}^2\backslash(\{t=0\}\cup\{s=0\}\cup\{s=t\}\cup\{s=-t\}).$$
We note that $\{t=0\}\cup\{s=0\}\cup\{s=t\}\cup\{s=-t\}$ is the disjoint union of four copies of $\mathbf{C}^*$ and one point, so we have that the Borel-Moore homology of $\tilde{\mathcal{Z}}$ is $\bar{H}_4(\tilde{\mathcal{Z}})=\mathbf{Q}(2),$ $\bar{H}_3(\tilde{\mathcal{Z}})=4\mathbf{Q}(1),$ $\bar{H}_2(\tilde{\mathcal{Z}})=3\mathbf{Q}$ and $\bar{H}_q(\tilde{\mathcal{Z}})=0$ for all $q\leq 0$ or $q\geq 5.$\\

To get the Borel-Moore homology of $\mathcal{Z},$ we need to consider first the following involutions of $\tilde{\mathcal{Z}}:$ 
\begin{itemize}
	\item[$i:$]$(t,s)\mapsto(s,t)$ exchanges the two points $S_1$ and $R_1$, hence the two conics;
	\item[$j:$]$(t,s)\mapsto(\frac{1}{t},\frac{1}{s})$ exchanges $0$ and $\infty$, therefore it acts as the involution on $l$ that exchanges the two tangency points;
	\item[$k:$]$(t,s)\mapsto(t,-s)$ exchanges $R_1$ with $R_2.$ (Note that $k$ has the same action on homology as $k':(t,s)\mapsto(-t,s)$ so we can consider only one of them).
\end{itemize}

By studying the action of $i,j,k$ on the stratification of $\mathbf{C}^2\backslash\tilde{\mathcal{Z}}$ into four copies of $\mathbf{C}^*$ and a point, we obtain
\begin{lem} The action of $i,j,k$ on the Borel-Moore homology classes of $\tilde{\mathcal{Z}}$ is as given in Table 1.\end{lem}
\begin{table}[H]\caption{}\footnotesize\label{table1}\centering
	\begin{tabular}{|c|c|c|c|}
		\hline 
		&$i$&$j$&$k$\\
		\hline
		degree 4&$+$&$+$ &$+$\\
		\hline
		degree 3&$+$&$+$&$+$\\
		&$+$&$+$&$-$\\
		&$+$&$-$&$+$\\
		&$-$&$-$&$+$\\
		\hline
		degree 2&$+$&$-$&$+$\\
		&$+$&$-$&$-$\\
		&$-$&$+$&$+$\\
		\hline
	\end{tabular}
\end{table}

Recall that by exchanging the two conics we are actually exchanging the points $R_1$ with $S_1,$ and $R_2$ with $S_2.$ So, we have to consider the invariant classes with respect to the involution $i.$ On the other hand, if we exchange $R_1$ with $R_2$, the two conics are not necessarily swapped. So, we also require the classes to be anti-invariant with respect to the action of $k.$

Thus, we get the Borel-Moore homology of $\mathcal{Z}$ and the spectral sequence of the bundle $X_{I+J}\rightarrow B$ will have two rows:
\begin{itemize}
	\item[] in degree 3: defined by the Borel-Moore homology of $B$ with constant coefficients;
	\item[] in degree 2: defined by the Borel-Moore homology of $B$ with non-constant coefficients $\mathcal{J}$.
	
\end{itemize}

To compute the Borel-Moore homology of the base space $B,$ we consider a covering $\tilde{B},$ where the points $A,B,C$ are ordered. Thus, there is a natural action of the symmetric group $\mathfrak{S}_3$ on $\tilde{B}$ and we can recover the Borel-Moore homology of $B$ by taking the $\mathfrak{S}_3\mbox{-}$ anti-invariant classes of the Borel-Moore homology of $\tilde{B}.$ \\

\begin{itemize}
	\item[$\bar{H}_\bullet(\tilde{B};\mathbf{Q}):$] Note that $\tilde{B}$ can be thought of as a fiber space over the space parametrizing three lines through $P,$ and the line $l,$ not passing through $P,$ that is $F(\mathbf{P}^1,3)\times \mathbf{C}^2.$ Denote by $r_A, r_B, r_C $ the three lines containing the points $A,B,C,$ respectively,
	After an appropriate change of coordinates, we may assume
	$$r_A: x_2=0,\qquad r_B:x_1=0,\qquad r_C:x_1-x_2=0,\qquad l:x=0.$$
	Then, the fiber of $\tilde{B}$ over $(r_A,r_B,r_C,l)$ is the space parametrizing the points $A,B,C$ and can be identified with a subset in $\mathbf{C}^3$: the point $(u,v,w)\in\mathbf{C}^3$ corresponds to the choices
	$$A=\left[1,u,0\right],\qquad B=\left[1,0,v\right],\qquad C=\left[1,w,w\right].$$
	We need then to remove the locus where the three points are collinear, which is a quadric cone of equation $uw+vw-uv=0.$\\
	Thus, $\bar{H}_\bullet(\tilde{B};\mathbf{Q})$ is invariant with respect to the involution $u\leftrightarrow v,$ and by noticing that this involution corresponds to the exchange of a couple of points among $A,B,C,$ it is also invariant with respect to the $\mathfrak{S}_3\mbox{-}$action.
	
	\item[$\bar{H}_\bullet(\tilde{B};\mathcal{J}):$] In order to compute the Borel-Moore homology of $\tilde{B}$ with non-constant coefficients, we will consider the subsets 
	$$\tilde{B}_J\overset{\text{open}}{\subseteq}\tilde{B}\qquad\text{and}\qquad\tilde{B}_I\overset{\text{closed}}{\subseteq}\tilde{B},$$
	where 
	$$\tilde{B}_J=\{((A,B,C),l)\in \tilde{B}|A,B,C\neq P\}$$
	defines configurations of type $(J),$ and 
	
	$$\tilde{B}_I=\{((A,B,C),l)\in\tilde{B}|\text{ one of $A,B,C$ is equal to $P$}\}$$
	defines those of type $(I).$ \\
	By projecting onto the triples $(A,B,C),$
	$$\tilde{B}_J\xrightarrow{Y_J} \{(A,B,C)|P\notin \overline{AB},\overline{BC},\overline{AC}; A,B,C\text{ not collinear}\},$$ where $Y_J\cong\mathbf{P}^2\backslash \{\text{four lines in general position}\},$ and by studying the preimage of $Y_J$ in the double cover of $\mathbf{P}^2$ ramified along four lines we notice that there is one only class in its Borel-Moore homology with non-constant coefficients, that is $\mathbf{Q}(1)$ in degree 2, and it is $\mathfrak{S}_3\mbox{-}$invariant.\\
	Similarly,	
	$$\tilde{B}_I\xrightarrow{Y_I} \{(A,B,C)|\text{ one of $A,B,C$ belongs to $E$}\},$$ where $Y_I\cong\mathbf{P}^2\backslash \{\text{three lines in general position, one with mult. 2}\}.$ \\
	Here, because of the double component, the covering is not normal, and its normalization is the double cover of $\mathbf{P}^2,$ ramified over the two simple lines. The fiber $Y_I$ and its double cover have the same homology with rational coefficients, thus the Borel-Moore homology with non-constant coefficients is trivial.\\
	Finally, by considering the Gysin exact sequence associated to inclusions	
	
	$$\tilde{B}_I\overset{\text{closed}}{\hookrightarrow}\tilde{B}\overset{\text{open}}\hookleftarrow \tilde{B}_J$$
	we get that $\bar{H}_\bullet(\tilde{B};\mathcal{J})$ also has no $\mathfrak{S}_3\mbox{-}$anti-invariant classes.
\end{itemize}
Since both homologies have only $\mathfrak{S}_3\mbox{-}$invariant classes, the Borel-Moore homology of $B$, both with constant and non-constant coefficients, will be trivial. Therefore we can conclude that the whole configuration space, and consequently $F_{I+J},$ has trivial Borel-Moore homologies.

\subsection{Column (L)}
Configurations of type $(L)$ are the singular loci of the blow up at the fixed point $P$ of a conic tangent to a line at $P$ and two other lines in the projective plane. Note that we can assume that the conic is irreducible, since we have already considered the reducible case that is config 76. If we define the space of configurations of the same type with the only exception that we let $P$ free in $\mathbf{P}^2:$ $\mathcal{L}:=\{(P,f)\in\mathbf{P}^2\times \Sigma| f\text{ has a node in }P$ and its singular points define a configuration of type $L \}$, where $\operatorname{dim}\mathcal{L}=20,$ then we can consider the space $X_L$ as the fiber of the bundle
\begin{align*}
\mathcal{L}&\rightarrow\mathbf{P}^2\\
(P,f)&\mapsto P.
\end{align*}
Let's consider such a configuration: in the projective plane, this is defined by the point $P$, and the intersection point of a conic $\mathcal{C}$ through $P$, its tangent at $P$, and 2 general lines $r,s$, not meeting at $P$ or any other point of the conic. 
We denote by $E_i, i=1,\dots,4$ the four points of intersection of the two lines and the conic, and $A,B$ the intersection points with the tangent line to the conic and we label the points as in the following figure.\\

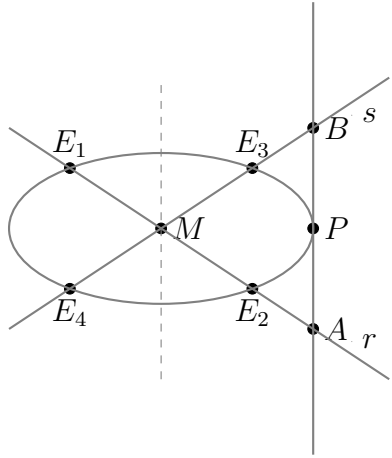
\begin{figure}[H]\centering
\begin{tikzpicture}
\filldraw[black] (6,0) circle (2pt) node[anchor=west] {$M$};
\filldraw[black] (8,0) circle (2pt) node[anchor=west] {$P$};
\filldraw[black] (8,1.3333) circle (2pt) node[anchor=west] {$B$};
\filldraw[black] (8,-1.3333) circle (2pt) node[anchor=west] {$A$};
\filldraw[black] (4.8,0.8) circle (2pt) node[anchor=south] {$E_1$};
\filldraw[black] (7.2,0.8) circle (2pt) node[anchor=south] {$E_3$};
\filldraw[black] (4.8,-0.8) circle (2pt) node[anchor=north] {$E_4$};
\filldraw[black] (7.2,-0.8) circle (2pt) node[anchor=north] {$E_2$};
\filldraw[] (8.5,1.5) circle (0pt) node[anchor=west] {$s$};
\filldraw[] (8.5,-1.5) circle (0pt) node[anchor=west] {$r$};
\draw[gray, thick] (4,1.3333) -- (9,-2);
\draw[gray, thick] (4,-1.3333) -- (9,2);
\draw[gray, thick] (8,-3) -- (8,3);
\draw[gray, thick] (6,0) ellipse (2 and 1);
\draw[gray, dashed] (6,-2) -- (6,2);
\end{tikzpicture}
\caption{configuration of type (L)}
\label{fig:config4}
\end{figure}

Up to projective transformations, we may assume the $E_i$ to be the projective frame of $\mathbf{P}^2:$ $E_1=\left[1,0,0\right],$ $E_2=\left[0,1,0\right],$ $E_3=\left[0,0,1\right],$ and $E_4=\left[1,1,1\right].$ Then we can consider another fiber bundle
$$\mathcal{L}\xrightarrow{PGL(3)} Y,$$
where $Y:=\{(P,A,B)|P\in \mathbf{P}^2\backslash \bigcup\overline{E_iE_j}; \{A,B\}=\mathcal{T}_P\mathcal{C}\cap(r\cup s)\}.$\\
Note that, once we have fixed the points $P,E_i$, $i=1,\dots,4$ and hence the lines $r,s$ and the conic, the points $A,B$ are uniquely determined. Thus, $Y$ is isomorphic to the space $\mathbf{P}^2\backslash \bigcup\overline{E_iE_j},$ that is isomorphic to the moduli space $\mathcal{M}_{0,5}$ of genus 0 curves with 5 marked points, since, for $n\geq3,$
$$\mathcal{M}_{0,n}=\{(t_0,\dots,t_{n-3})\in(\mathbf{P}^1)^{n-3}|t_i\neq 0,1,\infty,\text{ and }t_i\neq t_j\}.$$
By the equivariant Hodge Euler characteristic of $\mathcal{M}_{0,5}$ that is computed in \cite{getzler1995operads}, we have that the Borel-Moore homology of $\mathcal{M}_{0,5}$ is generated by the following classes:
\begin{itemize}
\item[]$\mathbf{Q}(2)\otimes\mathbf{S}_5$ in degree 4;
\item[]$\mathbf{Q}(1)\otimes\mathbf{S}_{3,2}$ in degree 3;
\item[]$\mathbf{Q}\otimes\mathbf{S}_{3,1^2}$ in degree 2;
\end{itemize}
where  by $-\otimes\mathbf{S}_{\lambda}$ we mean that we are considering the local system of coefficient corresponding to the irreducible representation of $\mathfrak{S}_5,$ associated to the partition $\lambda$ of 5.\\
On $Y$ there is a natural action of the dihedral group $D_4,$ that is the group of symmetries of a square, defined by the points $E_i,$ so, when computing its Borel-Moore homology, we need to consider local systems of coefficients defined by the action of $D_4,$ that can be embedded in the symmetric group $\mathfrak{S}_4$ by sending each symmetry to the corresponding permutation of vertices.
Restricting to $\mathfrak{S}_4$, we get the following representations:
\begin{align}\label{reps}
&\mathbf{S}_5\rightarrow\mathbf{S}_4\nonumber\\
&\mathbf{S}_{3,2}\rightarrow\mathbf{S}_{3,1}\oplus\mathbf{S}_{2,2}\\
&\mathbf{S}_{3,1^2}\rightarrow\mathbf{S}_{3,1}\oplus\mathbf{S}_{2,1^2}.\nonumber\end{align}
We then consider the character table of $D_4$, plus the lines of the character table of $\mathfrak{S}_4$ corresponding to the irreducible representations in \eqref{reps}, that can be found in \cite{Ser}:

\begin{center}
\begin{tabular}{ c|cccccc } 
	& $e$& (12)(34)&(1324)&(12) &(13)(24) &\\
	\hline
	$\psi_1$&1 &1 &1&1 &1 & \\
	$\psi_2$&1 &1 &1&-1 &-1 & \\
	$\psi_3$&1 &1 &-1&1 & -1& \\
	$\psi_4$&1 &1 &-1&-1 & 1& \\
	$\chi$&2 &-2 &0& 0&0 & \\
	\hline
	$\mathbf{S}_4$&1 &1 &1&1 &1 & =$\psi_1$\\
	$\mathbf{S}_{3,1}$&3 &-1 &-1&1 &-1 & =$\chi+\rho_{2}$\\
	$\mathbf{S}_{2,2}$&2 &2 &0&0 & 2& =$\psi_1+\rho_{3}$\\
	$\mathbf{S}_{2,1^2}$&3 &-1 &1&-1 & -1& =$\chi+\rho_{1}$\\
\end{tabular}
\end{center}

Hence we can write the Borel-Moore homology groups of $\mathcal{M}_{0,5}$ as $D_4\mbox{-}$ representations:
\begin{itemize}
\item[]$\mathbf{Q}(2)\otimes\psi_1$ in degree 4;
\item[]$\mathbf{Q}(1)\otimes(\psi_1+\rho_{2}+\rho_{3}+\chi)$ in degree 3;
\item[]$\mathbf{Q}\otimes(\rho_{1}+\rho_{2}+\chi^{\oplus2})$ in degree 2.
\end{itemize}
and we only need to consider the term involving the representation that corresponds to a local system of coefficients obtained by the restriction of $\pm\mathbf{Q}$ on $\pi_1(B(\mathbf{P}^2,8))=\mathfrak{S}_8$ to the fundamental group of our configuration space, represented in Figure \ref{fig:config4}. 
As we noticed before, the latter group is $D_4\subset \mathfrak{S}_8$ since it has to fix the points $P,M$ and the points $A,B$ are uniquely determined by the choice of the points $E_i$, whose permutations define the symmetric group $\mathfrak{S}_4\subset \mathfrak{S}_8.$ Because we also require $E_1,E_2\in r$ and $E_3,E_4\in s$ we get indeed $D_4.$ So the local system we are looking for is the restriction of the sign representation of $\mathfrak{S}_8$ to $D_4$ whose trace can be computed as follows.


\begin{itemize}
\item[$e$:] Clearly the identity will be mapped to +1;
\item[$(12)(34)$:] the element $(12)(34)$ acts by exchanging the two points on each of the two lines: $E_1\leftrightarrow E_2,$ $E_3\leftrightarrow E_4$ and thus will give a +1;
\item[$(1324):$] the element $(1324)$ corresponds to a rotation by $\pi/2$ of the $E_i$ that is an odd permutation of the $E_i$, but it also interchanges the two lines and hence the points $A,B,$ giving a +1;
\item[$(12):$] the element $(12)$ is the transposition of two points on the same line, moving no other point, so it will be mapped to -1;
\item[$(13)(24):$] finally, the element $(13)(24)$ corresponds to the symmetry with respect to the dashed line, that is an even permutation. This interchanges again the two lines, and hence $A,B$, giving -1.
\end{itemize}
By comparing this to the character table of $D_4,$ we get that the local system we want to consider is the one defined by the representation $\psi_2,$ that we will denote by $W$. Hence the Borel-Moore homology of $Y$ with coefficients in $W$ is $\mathbf{Q}$ in degree 2, and we can compute the Borel-Moore homology of $\mathcal{L}$ just by tensoring with the one of $PGL(3),$ that we can compute by duality from its cohomology: $\bar{H}_{16}(PGL(3,\mathbf{Q}))=\mathbf{Q}(8),$ $\bar{H}_{13}(PGL(3,\mathbf{Q}))=\mathbf{Q}(6),$ $\bar{H}_{11}(PGL(3,\mathbf{Q}))=\mathbf{Q}(5)$ and $\bar{H}_{8}(PGL(3,\mathbf{Q}))=\mathbf{Q}(3).$\\
Therefore, 
$\bar{H}_{18}(\mathcal{L};\mathbf{Q})=\mathbf{Q}(8),$ $\bar{H}_{15}(\mathcal{L};\mathbf{Q})=\mathbf{Q}(6),$ $\bar{H}_{13}(\mathcal{L};\mathbf{Q})=\mathbf{Q}(5),$ $\bar{H}_{10}(\mathcal{L};\mathbf{Q})=\mathbf{Q}(3),$ and it is zero in all the others degree.\\
Finally we compute the Borel-Moore homology of $X_L$ from the fibration
$$\mathcal{L}\rightarrow\mathbf{P}^2.$$
Since $\mathbf{P}^2$ is simply connected, there is a first quadrant spectral sequence
$$E^2_{p,q}=\bar{H}^p(\mathbf{P}^2)\otimes\bar{H}^q(X_L)\Rightarrow \bar{H}^{p+q}(\mathcal{L};\mathbf{Q}),$$

\begin{center}
\begin{tabular}{ c|ccccc } 
	14&$\mathbf{Q}(6)$\tikzmark{d} & &$\mathbf{Q}(7)$\tikzmark{b}& &$\mathbf{Q}(8)$  \\
	13&$\mathbf{Q}(5 )$ & &\tikzmark{c}$\mathbf{Q}(6)$& &\tikzmark{a}$\mathbf{Q}(7)$  \\
	12& & && & \\
	11&$\mathbf{Q}(4)$\tikzmark{h} & &$\mathbf{Q}(5)$\tikzmark{f}& & $\mathbf{Q}(6)$ \\
	10&$\mathbf{Q}(3)$ & &\tikzmark{g}$\mathbf{Q}(4)$& & \tikzmark{e}$\mathbf{Q}(5)$\\
	9& & & & & \\
	\hline
	&0&1&2&3&4
\end{tabular}
\begin{tikzpicture}[overlay, remember picture, yshift=.25\baselineskip, shorten >=.5pt, shorten <=.5pt]
\draw [shorten >=.1cm,shorten <=.1cm,->]([yshift=5pt]{pic cs:a}) -- ({pic cs:b});
\draw [shorten >=.1cm,shorten <=.1cm,->] ([yshift=5pt]{pic cs:c}) -- ({pic cs:d});
\draw [shorten >=.1cm,shorten <=.1cm,->]([yshift=5pt]{pic cs:e}) -- ({pic cs:f});
\draw [shorten >=.1cm,shorten <=.1cm,->]([yshift=5pt]{pic cs:g}) -- ({pic cs:h});
\end{tikzpicture}
\end{center}
where the differentials $d^2_{2,10}$, $d^2_{4,10}$, $d^2_{2,13}$, $d^2_{4,13}$ must be all isomorphisms in order to obtain the Borel-Moore homology of the total space that we computed above.
Therefore, since the space $F_L$ is a $\mathbf{C}\times\mathring{\Delta}_7\mbox{-}$bundle over $X_L,$ the Hodge-Grothendieck polynomial of $\bar{H}_{\bullet}(F_L;\mathbf{Q})$ must be $\mathbf{L}^{-7}t^{23}+\mathbf{L}^{-6}t^{22}+\mathbf{L}^{-5}t^{20}+\mathbf{L}^{-4}t^{19}.$

\subsection{Column (M)}
As a consequence of \cite[Lemma 2.10]{gorinov2005real}, $F_M$ is an open cone and its Borel-Moore homology can be obtained from the spectral sequence in Table \ref{table2}, whose columns coincide with those of the main spectral sequence, shifted by twice the dimension of the complex vector bundle that defines each column.\\
\begin{table}[H]\caption{}\small\label{table2}
\begin{tabular}{ c|ccccccccc} 
	12& & & & & &&&&$\mathbf{Q}(6)$\\
	11& & & & & &&&&$\mathbf{Q}(5)$\\
	10& & & & & &&&&\\
	9& & && & & &&&$\mathbf{Q}(4)$\\
	8& & && & & &&& $\mathbf{Q}(3)$\\
	7& & && & & &&&\\
	6& & && &&&&&\\
	5& & &&&&& & & \\
	4&&&& & && & & \\
	3& & &&$\mathbf{Q}(3)$ & & &$\mathbf{Q}(4)$\tikzmark{55}&\tikzmark{5}$\mathbf{Q}(4)$&\\
	2& &$\mathbf{Q}(2)$ && &$\mathbf{Q}(3)$ &$\mathbf{Q}(3)$ &&&\\
	1&$\mathbf{Q}(1)$ & &&$\mathbf{Q}(2)^2$ & & &$\mathbf{Q}(3)$&&\\
	0& &$\mathbf{Q}(1)$&$\mathbf{Q}(1)$& & &$\mathbf{Q}(2)$ &&&\\
	-1&$\mathbf{Q}$ & &&$\mathbf{Q}(1)$ & & &&&\\
	\hline
	&A&B&C&D&E&F&G&H&L\\
\end{tabular}
\begin{tikzpicture}[overlay, remember picture, yshift=.25\baselineskip, shorten >=.5pt, shorten <=.5pt]
\draw [shorten <=.1cm,->]([yshift=3pt]{pic cs:5}) -- ([yshift=3pt]{pic cs:55});
\end{tikzpicture}
\end{table}
We recall that $H^{\bullet}(V\backslash\Sigma)\cong\bar{H}_{35-\bullet}(\Sigma)$ and $V\backslash\Sigma$ is affine of dimension 18. Hence, for dimensional reasons, the differential $d^1_{H,3}:E^1_{H,3}\rightarrow E^1_{G,3}$ is non-trivial, and in the second page of the spectral sequence all differentials in the columns $(A)-(G)$ are isomorphisms.\\
We can then conclude that the Hodge-Grothendieck polynomial of $\bar{H}_{\bullet}(F_M,\mathbf{Q})$ is $\mathbf{L}^{-6}t^{22}+\mathbf{L}^{-5}t^{21}+\mathbf{L}^{-4}t^{19}+\mathbf{L}^{-3}t^{18}.$

\subsection{Spectral sequence}
The first page of the spectral sequence converging to $\bar{H}_{\bullet}(\Sigma,\mathbf{Q})$ is given in Table \ref{table3}. \\

	
\begin{table}[h]\caption{}\scriptsize\label{table3}	
	\begin{tabular}{ c|cccccccccc } 
		32& & $\mathbf{Q}(17)$&& & & &&&&\\
		31&$\mathbf{Q}(16)$ & && & & &&&&\\
		30& & $\mathbf{Q}(16)$&& & & &&&&\\
		29&$\mathbf{Q}(15)$ & && & & &&&&\\
		28& & && & & &&&&\\
		27& & && $\mathbf{Q}(15)$& & &&&&\\
		26& & &$\mathbf{Q}(14)$& & $\mathbf{Q}(15)$& &&&&\\
		25& & &&$\mathbf{Q}(14)^2$ & & &&&&\\
		24& & && & & &&&&\\
		23& & &&$\mathbf{Q}(13)$ & & &&&&\\
		22& & && & & $\mathbf{Q}(13)$&&&&\\
		21& & && & & &$\mathbf{Q}(13)$&&&\\
		20& & && & & $\mathbf{Q}(12)$&&&&\\
		19& & && & & &$\mathbf{Q}(12)$&&&\\
		18& & && & & &&&&\\
		17& & && & & &&$\mathbf{Q}(11)$&&\\
		16& & && & & &&&&\\
		15& & && & & &&&&\\
		14& & && &  &&&&$\mathbf{Q}(7)$&\\
		13& & && &  &&&&$\mathbf{Q}(6)$&\\
		12& & && &  &&&&&$\mathbf{Q}(6)$\\
		11& && & & &&&&$\mathbf{Q}(5)$&$\mathbf{Q}(5)$\\
		10& & & & &&&&&$\mathbf{Q}(4)$&\\
		9& &  & & &&&&&&$\mathbf{Q}(4)$\\
		8& &  & & &&&&&&$\mathbf{Q}(3)$\\
		\hline
		&A&B&C&D&E&F&G&H&L&M
	\end{tabular}
\end{table}

Following from Section 3.1, the cohomology of $X$ must contain a copy of the cohomology of $GL(2,\mathbf{C}).$ Applying then the isomorphism induced by the cap product with the fundamental class of the discriminant
$$\tilde{H}^{\bullet}(X;\mathbf{Q})\cong\bar{H}_{35-\bullet}(\Sigma;\mathbf{Q})(-d)$$
we compute the whole cohomology of $X$ and that of $X/GL(2,\mathbf{C}),$ whose Hodge-Grothendieck polynomial is $\mathbf{L}^{12}t^{13}+\mathbf{L}^{11}t^{12}+\mathbf{L}^4t^{6}+\mathbf{L}^3t^{5}+\mathbf{L}^2t^{3}+1.$

We finally consider the fibration
$$X/GL(2,\mathbf{C})\rightarrow X/(\mathbf{C}^*\times GL(2,\mathbf{C})).$$
There is a first quadrant cohomology spectral sequence starting with $E_2$ and converging to $H^{\bullet}(X/GL(2,\mathbf{C});\mathbf{Q})$:
$$E_2^{p,q}=H^p(X/(\mathbf{C}^*\times GL(2,\mathbf{C}));H^q(\mathbf{C}^*;\mathbf{Q}))\Rightarrow H^{p+q}(X/GL(2,\mathbf{C});\mathbf{Q})$$
and, because we know the cohomology of the total space and of the fibre, we can compute the cohomology of the base space from the second page of the spectral sequence represented in Table \ref{table4},

\begin{table}[h]\caption{}\footnotesize\label{table4}
\begin{tabular}{ c|ccccccccccccc} 
	1&$\mathbf{Q}(-1)$\tikzmark{w} &&$\mathbf{Q}(-2)$&&&$\mathbf{Q}(-4)$&&&&&&&$\mathbf{Q}(-12)$\\
	0&$\mathbf{Q}$ &&\tikzmark{ww}$\mathbf{Q}(-1)$&&&$\mathbf{Q}(-3)$&&&&&&&$\mathbf{Q}(-11)$\\
	\hline
	&0&1&2&3&4&5&6&7&8&9&10&11&12
\end{tabular}
\begin{tikzpicture}[overlay, remember picture, yshift=.25\baselineskip, shorten >=.5pt, shorten <=.5pt]
\draw [shorten >=.1cm,shorten <=.1cm,->]([yshift=5pt]{pic cs:w}) -- ([yshift=5pt]{pic cs:ww});
\end{tikzpicture}
\end{table}

where the differential $d_2^{0,1}:E_2^{0,1}\rightarrow E_2^{2,0}$ must be non-trivial since the term $\mathbf{Q}(-1)$ is not appearing in the cohomology of $X/GL(2,\mathbf{C}).$ So, the Hodge-Grothendieck polynomial of the cohomology of the base space, and hence of the moduli space of trigonal curve of genus 5, is $\mathbf{L}^{11}t^{12}+\mathbf{L}^3t^{5}+\mathbf{L}t^{2}+1.$
\clearpage

\appendix
\section{Trivial Configurations}

As we promised in the computation of the spectral sequences outlined in Table 2 and Table 3, we now consider the remaining configurations and prove that they have trivial twisted Borel-Moore homology.\\

\subsection{Configurations (42)-(43)}
Both these configurations are equivalent to the configurations of singularities of a plane quintic that is the union of a conic and a singular cubic. To be more precise, in the first configuration, the two curves meet each other at 6 distinct points and $P$ is any of the points of intersection, while in the second configuration they intersect at the singular point of the cubic, that is $P.$
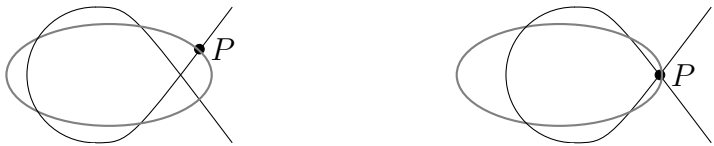
\begin{figure}[H]\centering
	\begin{tikzpicture}[scale = 0.9]
	
	\filldraw[black] (1.25,0) circle (2pt) node[anchor=west] {$P$};
	\draw (0,-1) .. controls (0.5,-1)  .. (2,1);
	\draw (0,1) .. controls (0.5,1) .. (2,-1);
	\draw[gray, thick] (-0.22,0) ellipse (1.5 and 0.75);
	\begin{scope}
	\clip (-1.5,-1) rectangle (0,1.5);
	\draw (0,0) circle(1);
	\end{scope}

	\filldraw[black] (-5.48,0.38) circle (2pt) node[anchor=west] {$P$};
	\draw (-7,-1) .. controls (-6.5,-1)  .. (-5,1);
	\draw (-7,1) .. controls (-6.5,1) .. (-5,-1);
	\draw[gray, thick] (-6.8,0) ellipse (1.5 and 0.75);
	\begin{scope}
	\clip (-8.5,-1) rectangle (-7,1.5);
	\draw (-7,0) circle(1);
	\end{scope}
	\end{tikzpicture}
\end{figure}
Both configuration spaces can be fibered over the space of conics through $P.$ If we denote the conic by $\mathcal{C},$ the fibers will be respectively equal to $B(\mathcal{C}\backslash\{P\},5)$ and $B(\mathcal{C}\backslash\{P\},4)$, which both have trivial twisted Borel-More homology by Lemma 2.1.

\subsection{Configurations (44)-(45)-(46)}
These configurations are all obtained by blowing up a singular point in the configuration of type 37 in \cite{gorinov2005real}, defined by the intersection points of two lines and a cubic curve in $\mathbf{P}^2$ having one singular point. 
To be more precise, configurations of these types correspond to the blow up at $P$, where $P$ has to be an ordinary double point: it is first defined as the point of intersection between a line and the cubic, then as the point of intersection between the two lines, and finally as the singular point of the cubic. Note that, in the first two configuration spaces, the cubic need not to be irreducible: it can decompose into three concurrent lines or into the union of a conic and a line tangent to it. However, this cannot happen for configuration (46), otherwise $P$ would not be a double ordinary singularity. The two reducible cases define configurations (59) and (55), respectively. Configuration (55) was already considered as configuration ($L$), while configuration (59) will be considered later.\\
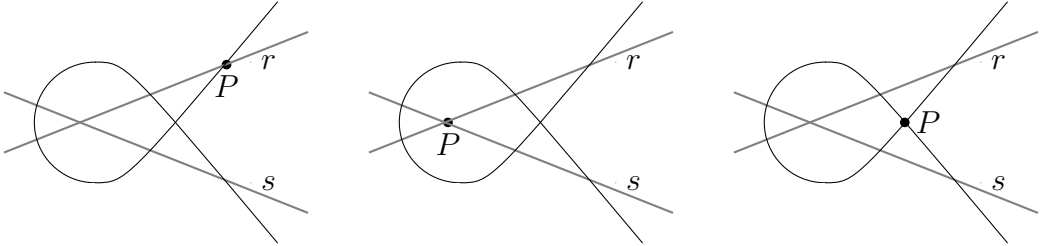
\begin{figure}[h!]\centering
	\begin{tikzpicture}[scale = 0.8]
	\filldraw[] (-3.45,1) circle (0pt) node[anchor=west] {$r$};
	\filldraw[] (-3.45,-1) circle (0pt) node[anchor=west] {$s$};
	
	\filldraw[black] (-3.84,0.958) circle (2pt) node[anchor=north] {$P$};
	\draw[gray, thick] (-7.5,-0.5) -- (-2.5,1.5);
	\draw[gray, thick] (-7.5,0.5) -- (-2.5,-1.5);
	\draw (-6,-1) .. controls (-5.5,-1)  .. (-3,2);
	\draw (-6,1) .. controls (-5.5,1) .. (-3,-2);
	\begin{scope}
	\clip (-7.5,-1) rectangle (-6,1.5);
	\draw (-6,0) circle(1);
	\end{scope}
	
	\filldraw[] (2.55,1) circle (0pt) node[anchor=west] {$r$};
	\filldraw[] (2.55,-1) circle (0pt) node[anchor=west] {$s$};
	\filldraw[black] (-0.2,0) circle (2pt) node[anchor=north] {$P$};
	\draw[gray, thick] (-1.5,-0.5) -- (3.5,1.5);
	\draw[gray, thick] (-1.5,0.5) -- (3.5,-1.5);
	\draw (0,-1) .. controls (0.5,-1)  .. (3,2);
	\draw (0,1) .. controls (0.5,1) .. (3,-2);
	\begin{scope}
	\clip (-1.5,-1) rectangle (0,1.5);
	\draw (0,0) circle(1);
	\end{scope}
	\filldraw[] (8.55,1) circle (0pt) node[anchor=west] {$r$};
	\filldraw[] (8.55,-1) circle (0pt) node[anchor=west] {$s$};
	\filldraw[black] (7.31,0) circle (2pt) node[anchor=west] {$P$};
	\draw[gray, thick] (4.5,-0.5) -- (9.5,1.5);
	\draw[gray, thick] (4.5,0.5) -- (9.5,-1.5);
	\draw (6,-1) .. controls (6.5,-1)  .. (9,2);
	\draw (6,1) .. controls (6.5,1) .. (9,-2);
	\begin{scope}
	\clip (4.5,-1) rectangle (6,1.5);
	\draw (6,0) circle(1);
	\end{scope}
	\end{tikzpicture}
\end{figure}

Configuration spaces of type (44), (45), (46) can then all be fibered over the space parametrizing the two lines $r,s$, with fibers isomorphic to the quotient of $B(\mathbf{C}^*,2)\times B(\mathbf{C},3)$, $B(\mathbf{C},3)\times B(\mathbf{C},3)$ and $B(\mathbf{C},3)\times B(\mathbf{C},3)$, respectively, by the involution given by exchanging the two lines. Because the fibers have all trivial twisted Borel-Moore homology, the homology of the configuration spaces will be trivial as well. \\

\subsection{Configuration (49)}
Configurations of type $(49)$ are obtained by the same plane curve considered in those of type $(J),$ where $P$ is defined as the point of intersection of a conic and a line.\\
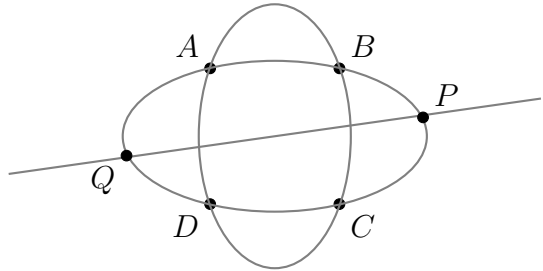
\begin{figure}[H]\centering
	\begin{tikzpicture}
	\filldraw[black] (13.85,0.9) circle (2pt) node[anchor=south west] {$B$};
	\filldraw[black] (12.15,0.9) circle (2pt) node[anchor=south east] {$A$};
	\filldraw[black] (13.85,-0.9) circle (2pt) node[anchor=north west] {$C$};
	\filldraw[black] (12.15,-0.9) circle (2pt) node[anchor=north east] {$D$};
	\draw[gray, thick] (9.5,-0.5) -- (16.5,0.5);
	\draw[gray, thick] (13,0) ellipse (2 and 1);
	\draw[gray, thick] (13,0) ellipse (1 and 1.75);
	\filldraw[black] (14.95,0.25) circle (2pt) node[anchor=south west] {$P$};
	\filldraw[black] (11.05,-0.257) circle (2pt) node[anchor=north east] {$Q$};
	\end{tikzpicture}
	\label{fig:config3}
\end{figure} Then $X_K$ can be fibered over $\tilde{B}(\mathbf{P}^2\backslash\{P\},4)\ni\{A,B,C,D\}.$ Once these points are fixed, we notice that the conic $\mathcal{C}$ passing also through $P$ is uniquely determined. Therefore the fiber $Y$ is itself a fiber bundle over $L\cong\mathbf{P}^1\backslash\{\text{5 points}\},$ the space of lines not passing through any of the points $A,B,C,D$ and not tangent to $\mathcal{C}$ with fiber $\mathcal{Z}$ defined as the space of conics not tangent to $l\in L$ and different from $\mathcal{C}.$\\
$\mathcal{Z}$ is isomorphic to $\mathbf{P}^1\backslash\{0,1,\infty\}\cong\mathbf{C}\backslash\{0,1\},$ and thus, since determining a conic in $\mathcal{Z}$
is equivalent to choosing a point in $l$ that is different to $P,Q$ and the $2$ points of tangency $T_1,T_2$ in $l,$ $\bar{H}_{\bullet}(\mathcal{Z},\pm\mathbf{Q})$ is $\mathbf{Q}=\left[T_1-T_2\right]$ in degree 1 and 0 in all other degrees. Note also that, when moving $l$ around $A,$ for instance, the points of tangency in $l$ are swapped. Therefore $\pi_1(L)$ acts on $\left[T_1-T_2\right]$ anti-invariantly and the Borel-Moore homology of the fiber $Y$ is defined by that of $L$ with non-trivial coefficient system:
$$\bar{H}_{\bullet}(L;\bar{H}_1(\mathcal{Z}))=\mathbf{Q}, \qquad\text{in degree }0.$$
Finally, we notice that, because we are considering a local system on $L$ that changes its sign under the action of any loop in $\mathbf{P}^1$ around any point removed, then any $\gamma\in\tilde{B}(\mathbf{P}^2\backslash\{P\},4)$ transposing a pair of points must act on the fiber as the  multiplication by -1. Therefore the local system induced by the fiber on $\tilde{B}(\mathbf{P}^2\backslash\{P\},4)$ is $\pm\mathbf{Q}$ and by Lemma 2.3 the twisted Borel-Moore homology of $X_K$ will be trivial.
\subsection{Configurations (50)-(51)} 
In all these configuration spaces $P$ has to be a triple point. More precisely, they are the defined by blowing up the following curves at $P.$  
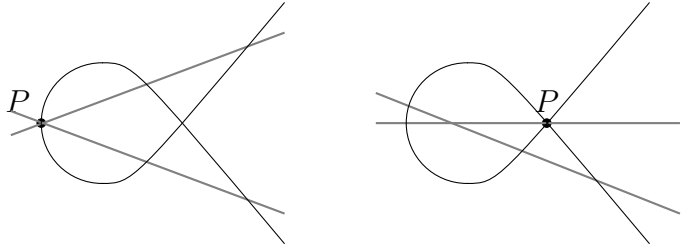
\begin{figure}[h!]\centering
	\begin{tikzpicture}[scale = 0.8]
	\filldraw[black] (-7,0) circle (2pt) node[anchor=south east] {$P$};
	\draw[gray, thick] (-7.5,-0.2) -- (-3,1.5);
	\draw[gray, thick] (-7.5,0.2) -- (-3,-1.5);
	\draw (-6,-1) .. controls (-5.5,-1)  .. (-3,2);
	\draw (-6,1) .. controls (-5.5,1) .. (-3,-2);
	\begin{scope}
	\clip (-7.5,-1) rectangle (-6,1.5);
	\draw (-6,0) circle(1);
	\end{scope}
	
	\filldraw[black] (1.31,0) circle (2pt) node[anchor=south] {$P$};
	\draw[gray, thick] (-1.5,0) -- (3.5,0);
	\draw[gray, thick] (-1.5,0.5) -- (3.5,-1.5);
	\draw (0,-1) .. controls (0.5,-1)  .. (3,2);
	\draw (0,1) .. controls (0.5,1) .. (3,-2);
	\begin{scope}
	\clip (-1.5,-1) rectangle (0,1.5);
	\draw (0,0) circle(1);
	\end{scope}
	\end{tikzpicture}
\end{figure}

We can fiber the spaces over the space parametrizing the pairs of lines. The fiber spaces will be then isomorphic to a quotient of $B(\mathbf{C},2)\times B(\mathbf{C},2)$ and $\mathbf{C}^*\times B(\mathbf{C},3),$ respectively, and they both have trivial twisted Borel-Moore homology. 

\subsection{Configuration (52)}
As above, $P$ must be again a triple point. In particular, configurations of type (52) are defined by two distinct conics meeting at $P$ and three additional points $A,B,C,$ and a line $l$ through $P,$ not meeting any of $A,B,C.$\\

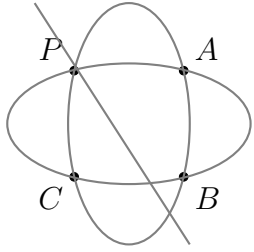
\begin{figure}[H]\centering
	\begin{tikzpicture}[scale = 0.8]
	\filldraw[black] (6.1,0.88) circle (2pt) node[anchor=south east] {$P$};
	\filldraw[black] (7.9,0.88) circle (2pt) node[anchor=south west] {$A$};
	\filldraw[black] (6.1,-0.88) circle (2pt) node[anchor=north east] {$C$};
	\filldraw[black] (7.9,-0.88) circle (2pt) node[anchor=north west] {$B$};
	\draw[gray, thick] (5.45,2) -- (8,-2);
	
	\draw[gray, thick] (7,0) ellipse (2 and 1);
	\draw[gray, thick] (7,0) ellipse (1 and 2);
	
	\end{tikzpicture}
\end{figure}

Then, the configuration space can be fiber over the space $\tilde{B}(\mathbf{P}^2\backslash\{P\},3)\times (\mathbf{P}^1\backslash\{\text{3 points}\})\ni\{(\{A,B,C\},l)\},$ parametrizing the intersection points between the two conics and the choices for the line $l.$ Once we have fixed $l,$ two points on it will uniquely determine the two conics. Hence, the fiber space will be $B(\mathbf{C},2)$ whose Borel-Moore homology will be considered with constant coefficient because, when we exchange the two conics we are actually exchanging 2 couples of points in the configuration space: the two points lying on the line, and the two points of intersection between the exceptional divisor and the strict transforms of the two conics. On the other hand, there is a natural action of $\mathfrak{S}_3$ on the base space, and by noticing that both factors have no $\mathfrak{S}_3\mbox{-}$anti-invariant classes in their homologies\footnotemark\footnotetext{For the factor $\tilde{B}(\mathbf{P}^2\backslash\{P\},3)$ this follows by Lemma 2.3. While for the second factor, this can be deduced by computing the Borel-Moore homology of $\mathbf{P}^1\backslash\{\text{3 points}\}$ in terms of $\mathfrak{S}_3\mbox{-}$ representations, that is $\mathbf{S}_3(1)$ in degree 2 and $\mathbf{S}_{2,1}$ in degree 1.} the total space will have trivial twisted Borel-Moore homology.
\subsection{Configurations (53)-(54)}
These configurations are obtained by blowing up respectively a point of intersection between two lines and a point of intersection between a line and a conic in the set of singular points in $\mathbf{P}^2$ defined by 3 points $A,B,C$ in general position + 6 points of intersection between the three lines $\overline{AB},\overline{BC},\overline{AC}$ and a conic not tangential to the lines, that is config. 39 in \cite{gorinov2005real}.\\
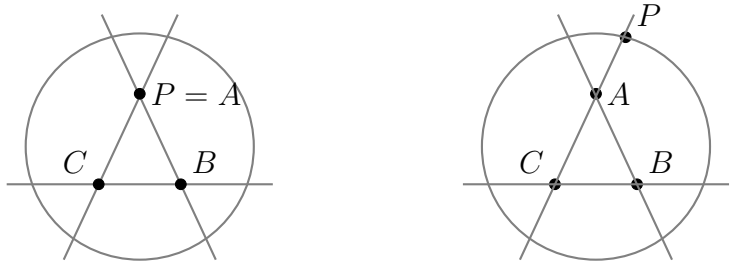
\begin{figure}[H]\centering
	\begin{tikzpicture}
	
	\filldraw[black] (6.39,1.45) circle (2pt) node[anchor=south west] {$P$};
	\draw[gray, thick] (-1,-1.5) -- (0.5,1.75);
	\draw[gray, thick] (1,-1.5) -- (-0.5,1.75);
	\draw[gray, thick] (-1.75,-0.5) -- (1.75,-0.5);
	\draw[gray, thick] (0,0) circle (1.5); 
	
	\filldraw[black] (6.54,-0.5) circle (2pt) node[anchor=south west] {$B$};
	\filldraw[black] (5.46,-0.5) circle (2pt) node[anchor=south east] {$C$};
	
	\filldraw[black] (0,0.7) circle (2pt) node[anchor=west] {$P=A$};
	\filldraw[black] (-0.54,-0.5) circle (2pt) node[anchor=south east] {$C$};
	\filldraw[black] (0.54,-0.5) circle (2pt) node[anchor=south west] {$B$};
	\filldraw[black] (6,0.7) circle (2pt) node[anchor=west] {$A$};
	\draw[gray, thick] (5,-1.5) -- (6.5,1.75);
	\draw[gray, thick] (7,-1.5) -- (5.5,1.75);
	\draw[gray, thick] (4.25,-0.5) -- (7.75,-0.5);
	\draw[gray, thick] (6,0) circle (1.5);
	\end{tikzpicture}
\end{figure}

As in \cite{gorinov2005real}, we want to fiber both the configuration spaces over the spaces parametrizing the points of intersection between the three lines.\\
When we choose $P$ as one of these points, e.g. $A,$ the total space will be fibered over $B(\mathbf{P}^2\backslash\{P\},2)$ instead of $B(\mathbf{P}^2,3).$\\
On the other hand, when $P$ is the intersection point between the conic and a line, the configuration space is fibered over a quotient of $B(\mathbf{P}^2\backslash\{P\},2)\times \mathbf{C}^*\ni(\{B,C\},A).$\\
The fiber space, denoted by $Y$ in \cite{gorinov2005real}, will be in both cases the same, i.e. a fiber bundle over $B(\mathbf{C}^*,2)\times B(\mathbf{C}^*,2)$, the configuration space of 2 points on each of $\overline{AB},\overline{BC}$, excluding $A,B,C.$ Therefore it will have the same Borel-Moore homology, that is $\mathbf{Q}$ in degree 5, $\mathbf{Q}(1)^2$ in degree 6 and $\mathbf{Q}(2)$ in degree 7, but we will have to consider the action of the fundamental group of the new base space that is either $B(\mathbf{P}^2\backslash\{P\},2)$, or it contains it as a factor of a product. The fundamental group will then be the restriction of the symmetric group $\mathfrak{S}_3$ to $\{B,C\}$: $\mathfrak{S}_2.$\\
Thus, we only need to consider local systems of coefficients corresponding to the restrictions of the representations of $\mathfrak{S}_3$: trivial and sign representation will restrict respectively to trivial and sign representation on $\mathfrak{S}_2,$ while the 2-dimensional irreducible representation restricts to the direct sum of the trivial and sign representation. We have that 
$\bar{P}(B(\mathbf{P}^2\backslash\{P\},2),\mathbf{Q})=\mathbf{L}^{-4}t^8$ and $\bar{P}(B(\mathbf{P}^2\backslash\{P\},2),\pm\mathbf{Q})=\mathbf{L}^{-3}t^6,$ so we will get a similar $E^2-$term of the spectral sequence, with the only difference that the action of $\mathfrak{S}_2$ on $\bar{H}_6(Y;\pm\mathbf{Q})=\mathbf{Q}(1)^2$ now must be reducible:

\begin{center}
	\begin{tabular}{ c|cccccc } 
		7&$\mathbf{Q}(5)$\tikzmark{B} & && & \\
		
		6& $\mathbf{Q}(4)$\tikzmark{D}& &\tikzmark{A}$\mathbf{Q}(5)$& & & \\
		
		5& & &\tikzmark{C}$\mathbf{Q}(4)$& & \\
		
		\hline
		&6&7&8&9&10\\
	\end{tabular}
	\begin{tikzpicture}[overlay, remember picture, yshift=.25\baselineskip, shorten >=.5pt, shorten <=.5pt]
	\draw [shorten >=.1cm,shorten <=.1cm,->]([yshift=5pt]{pic cs:A}) -- ({pic cs:B});
	\draw [shorten >=.1cm,shorten <=.1cm,->] ([yshift=5pt]{pic cs:C}) -- ({pic cs:D});
	\end{tikzpicture}
\end{center}
and the differentials $d^2_{8,5}:E^2_{8,5}\rightarrow E^2_{6,6},$ $d^2_{8,6}:E^2_{8,6}\rightarrow E^2_{6,7}$ must be isomorphisms since both $E^2_{5,8},E^2_{6,6}$ are generated by $\bar{H}_6(B(\mathbf{P}^2\backslash\{P\},2),\pm\mathbf{Q})$ and both $E^2_{6,8},E^2_{7,6}$ are generated by $\bar{H}_8(B(\mathbf{P}^2\backslash\{P\},2),\mathbf{Q}).$\\
Therefore, also these configuration spaces have trivial twisted Borel-Moore homology.

\subsection{Configuration (56)-(57)} As configurations of type (50),(51) and (52), here $P$ is also a triple point, but we have one additional singular point. In these two cases, we obtain the configurations of singularities by blowing up the following curves at $P.$\\

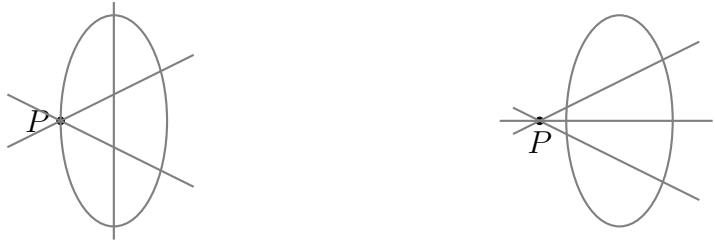
\begin{figure}[H]\centering
	\begin{tikzpicture}[scale = 0.7]
	\filldraw[black] (0,0) circle (2pt) node[anchor=east] {$P$};
	\draw[gray, thick] (-1,-0.5) -- (2.5,1.25);
	\draw[gray, thick] (-1,0.5) -- (2.5,-1.25);
	\draw[gray, thick] (1,2.25) -- (1,-2.25);
	\draw[gray, thick] (1,0) ellipse (1 and 2);
	
	\filldraw[black] (9,0) circle (2pt) node[anchor=north] {$P$};
	\draw[gray, thick] (8.5,-0.25) -- (12,1.5);
	\draw[gray, thick] (8.5,0.25) -- (12,-1.5);
	\draw[gray, thick] (8.25,0) -- (12.25,0);
	\draw[gray, thick] (10.5,0) ellipse (1 and 2);
	
	\end{tikzpicture}
\end{figure}

Consider first configurations of type (56). Similarly to configurations (53), (54), we fiber the space over the points of intersection between the lines, that is $\tilde{B}(\mathbf{P}^2\backslash\{P\},2).$
These two points, together with $P,$ uniquely determine the three lines, and the 4 points left defining the configuration, together with $P$, will uniquely determine the conic. The fiber space is then isomorphic to the quotient of  $\mathbf{C}^*\times\mathbf{C}^*\times B(\mathbf{C}^*,2)$ by the involution exchanging the first two factors, thus it has twisted Borel-Moore homology equal to $\mathbf{Q}$ in degree $4$ and $\mathbf{Q}(1)$ in degree 5.
The fundamental group of the base space, acts by exchanging the two points, and thus induces an $\mathfrak{S}_2\mbox{-}$action on the line not passing through $P$ that is the one described in Lemma 2.2. Therefore, the Borel-Moore homology class in degree $5$ must be anti-invariant under such action, while the class in degree $4$ will be invariant and by applying Lemmas 2.3, 2.4 and 2.6 we have that the second page of the spectral sequence must have the following form:
\begin{center}
	\begin{tabular}{ c|ccc } 
		5&$\mathbf{Q}(4)$\tikzmark{end}& &\\
		4&& &\tikzmark{start}$\mathbf{Q}(4)$\\
		\hline
		&6&7&8\\
	\end{tabular}
	\begin{tikzpicture}[overlay, remember picture, yshift=.25\baselineskip, shorten >=.5pt, shorten <=.5pt]
	\draw [shorten >=.1cm,shorten <=.1cm,->]([yshift=5pt]{pic cs:start}) -- ({pic cs:end});
	\end{tikzpicture}
\end{center}
where the differential will be an isomorphism.\\
On the other hand, we can fiber the second configuration space over the space parametrizing the three lines through $P.$ It suffices to fix 5 points on those lines to determine the conic, therefore the fiber space will be a quotient of $B(\mathbf{C},2)\times B(\mathbf{C},2)\times \mathbf{C},$ so this configuration space will also give no contribution to the Borel-Moore homology of the discriminant.

\subsection{Configurations (58)-(59)}
These configuration spaces are both defined by 5 lines in the projective plane. The first is obtained by blowing up any of the singular points of 5 lines in general position, while the second one is obtained by blowing up the point of intersection of three concurrent lines, in a plane quintic define by those lines and two additional lines meeting at a point outside the three concurrent lines. 

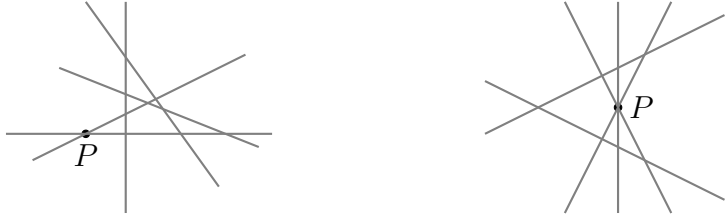
\begin{figure}[h!]\centering
	\begin{tikzpicture}[scale = 0.7]
	\filldraw[black] (-1,-0.5) circle (2pt) node[anchor=north] {$P$};
	\draw[gray, thick] (-2,-1) -- (2,1);
	\draw[gray, thick] (-0.25,-2) -- (-0.25,2);
	\draw[gray, thick] (-2.5,-0.5) -- (2.5,-0.5);
	\draw[gray, thick] (-1,2) -- (1.5,-1.5);
	\draw[gray, thick] (-1.5,0.75) -- (2.25,-0.75);

	\filldraw[black] (9,0) circle (2pt) node[anchor=west] {$P$};
	\draw[gray, thick] (9,2) -- (9,-2);
	\draw[gray, thick] (8,2) -- (10,-2);
	\draw[gray, thick] (8,-2) -- (10,2);
	\draw[gray, thick] (6.5,0.5) -- (11,-1.75);
	\draw[gray, thick] (6.5,-0.5) -- (11,1.75);

	\end{tikzpicture}
\end{figure}

We can fiber both configuration spaces over the set of lines meeting at $P,$ thus:
$$X_{(58)}\rightarrow B(\mathbf{P}^1,2)\qquad\text{and}\qquad X_{(59)}\rightarrow B(\mathbf{P}^1,3),$$
the fiber space will then be the space of the remaining lines defining the configuration that are, respectively, $B(\mathbf{C}^2,3)$ and $B(\mathbf{C}^2,2)$ by duality, and both have trivial twisted Borel-Moore homology by Lemma 2.1.
\clearpage

\bibliographystyle{alpha}

\begin{thebibliography}{99}
	
	\bibitem[BV12]{bolognesi2012stacks}
	Michele Bolognesi and Angelo Vistoli.
	\newblock Stacks of trigonal curves.
	\newblock {\em Transactions of the American Mathematical Society},
	364(7):3365--3393, 2012.
	
	\bibitem[CGP18]{chan2018tropical}
	Melody Chan, Soren Galatius, and Sam Payne.
	\newblock Tropical curves, graph homology, and top weight cohomology of
	$\mathcal{M}_g$.
	\newblock {\em arXiv preprint arXiv:1805.10186}, 2018.
	
	\bibitem[CGP19]{chan2019topology}
	Melody Chan, Soren Galatius, and Sam Payne.
	\newblock Topology of moduli spaces of tropical curves with marked points.
	\newblock {\em arXiv preprint arXiv:1903.07187}, 2019.
	
	\bibitem[Get95]{getzler1995operads}
	Ezra Getzler.
	\newblock Operads and moduli spaces of genus 0 riemann surfaces.
	\newblock In {\em The moduli space of curves}, pages 199--230. Springer, 1995.
	
	\bibitem[Gor05]{gorinov2005real}
	Alexei~G Gorinov.
	\newblock Real cohomology groups of the space of nonsingular curves of degree 5
	in $\mathbb{C}\mathbb{P}^2$.
	\newblock In {\em Annales de la Facult{\'e} des sciences de Toulouse:
		Math{\'e}matiques}, volume~14, pages 395--434, 2005.
	
	\bibitem[Loo93]{Loo}
	Eduard Looijenga.
	\newblock Cohomology of $\mathcal{M}_3$ and $\mathcal{M}_3^1$.
	\newblock {\em Mapping class groups and moduli spaces of Riemann surfaces},
	150:205--228, 1993.
	
	\bibitem[Mar46]{Mar}
	Arturo Maroni.
	\newblock Le serie lineari speciali sulle curve trigonali.
	\newblock {\em Annali di Matematica Pura ed Applicata}, 25(1):341--354, 1946.
	
	\bibitem[PS03]{PS}
	Chris~AM Peters and Joseph~HM Steenbrink.
	\newblock Degeneration of the leray spectral sequence for certain geometric
	quotients.
	\newblock {\em Moscow Mathematical Journal}, 3(3):1085--1095, 2003.
	
	\bibitem[PV15]{PV}
	Anand Patel and Ravi Vakil.
	\newblock On the chow ring of the hurwitz space of degree three covers of
	$\mathbb{P}^1$.
	\newblock {\em arXiv preprint arXiv:1505.04323}, 2015.
	
	\bibitem[Ser77]{Ser}
	Jean-Pierre Serre.
	\newblock {\em Linear representations of finite groups}, volume~42.
	\newblock Springer, 1977.
	
	\bibitem[SF97]{Sta}
	Zvezdelina~E Stankova-Frenkel.
	\newblock Moduli of trigonal curves.
	\newblock {\em arXiv preprint alg-geom/9710015}, 1997.
	
	\bibitem[Tom05]{Tom}
	Orsola Tommasi.
	\newblock Rational cohomology of the moduli space of genus 4 curves.
	\newblock {\em Compositio Mathematica}, 141(2):359--384, 2005.
	
	\bibitem[Vas99]{Vart}
	V.~A. Vasil{\cprime}ev.
	\newblock How to calculate the homology of spaces of nonsingular algebraic
	projective hypersurfaces.
	\newblock {\em Tr. Mat. Inst. Steklova}, 225(Solitony Geom. Topol. na
	Perekrest.):132--152, 1999.
	\newblock arXiv:1407.7229 [math.AG].
	
	\bibitem[Wen20]{Wen}
	Thomas Wennink.
	\newblock Counting the number of trigonal curves of genus 5 over finite fields.
	\newblock {\em Geometriae Dedicata}, pages 1--18, 2020.
	
\end{thebibliography}
\def\cprime{$'$}

\end{document}